\documentclass{amsart}[12 pt]

\usepackage{latexsym}

\usepackage{amsmath,amssymb, amsfonts,color}
\usepackage{verbatim}
\usepackage{graphicx}
\usepackage{graphics}
\usepackage{geometry}
\usepackage{booktabs}

\newtheorem{proposition}{Proposition}[section]
\newtheorem{theorem}{Theorem}[section]
\newtheorem{definition}{Definition}[section]
\newtheorem{corollary}{Corollary}[section]
\newtheorem{lemma}{Lemma}[section]
\newtheorem{remark}{Remark}[section]
\newtheorem{example}{Example}[section]

 \newcommand{\Dl}{\Delta}

 \newcommand{\om}{\omega}
  
 \newcommand{\al}{\alpha}
 \newcommand{\ffi}{\varphi}
 \newcommand{\gt}{\geqslant}
 \newcommand{\td}{\tilde}
  \newcommand{\gm}{\gamma}
  
\newcommand{\lt}{\leqslant}
\newcommand{\nab}{\nabla}
\newcommand{\mto}{\mapsto}
\newcommand{\eps}{\varepsilon}
\newcommand{\ind}{\mathbb I}
\renewcommand{\le}{\lt}
\renewcommand{\ge}{\gt}

\def\Rnu{{\mathbb R}}

\def\aa#1{ \begin{align*} #1 \end{align*} }
\def\aaa#1{ \begin{align} #1 \end{align} }
\def\mm#1{ \begin{multline*} #1 \end{multline*} }
\def\mmm#1{ \begin{multline} #1 \end{multline} }

\numberwithin{equation}{section}

\allowdisplaybreaks

\begin{document}
\markboth{}{}

\title[Path-dependent It\^o formulas under $(p,q)$-variations] {Path-dependent It\^o formulas under finite $(p,q)$-variation regularity}

%\author{Dorival Le\~ao}

%\address{Departamento de Matem\'atica Aplicada e Estat\'istica. Universidade de S\~ao
%Paulo, 13560-970, S\~ao Carlos - SP, Brazil} \email{leao@icmc.usp.br}

%\author{Dorival Le\~ao}

\author{Alberto Ohashi}

\address{Departamento de Matem\'atica, Universidade Federal da Para\'iba, 13560-970, Jo\~ao Pessoa - Para\'iba, Brazil}\email{alberto.ohashi@pq.cnpq.br; ohashi@mat.ufpb.br}

\author{Evelina Shamarova}

\address{Departamento de Matem\'atica, Universidade Federal da Para\'iba, 13560-970, Jo\~ao Pessoa - Para\'iba, Brazil}\email{evelina@mat.ufpb.br}

\author{Nikolai N. Shamarov}

\address{Lomonosov Moscow State University, Faculty of Mechanics and Mathematics, 119991,
Vorob'evy Gory, Moscow,  Russia}\email{nshamarov@yandex.ru}

\date{\today}

\keywords{Functional stochastic calculus, 2D Young integrals, local-time} \subjclass{Primary: C02; Secondary: G12}

\begin{center}

\end{center}

\begin{abstract}
In this work, we establish pathwise functional It\^o formulas for non-smooth functionals of real-valued continuous semimartingales. Under finite $(p,q)$-variation regularity assumptions in the sense of two-dimensional Young integration theory, we establish a pathwise local-time decomposition

$$F_t(X_t) = F_0(X_0)+ \int_0^t\nabla^hF_s(X_s)ds + \int_0^t\nabla^wF_s(X_s)dX(s) - \frac{1}{2}\int_{-\infty}^{+\infty}\int_0^t(\nabla^w_xF_s)(^{x}X_s)d_{(s,x)}\ell^x(s).$$
Here, $X_t= \{X(s); 0\le s\le t\}$ is the continuous semimartingale path up to time $t\in [0,T]$, $\nabla^h$
is the horizontal derivative, $(\nabla^w_x F_s)(^{x}X_s)$ is a weak derivative of $F$ with respect to the terminal value $x$ of the modified path $^{x}X_s$
and $\nabla^w F_s(X_s) = (\nabla^w_x F_s)(^{x} X_s)|_{x=X(s)}$.
The double integral is interpreted as a space-time 2D-Young integral with differential $d_{(s,x)}\ell^x(s)$, where $\ell$ is the local-time of $X$. Under less restrictive joint variation assumptions on $(\nabla^w_x F_t)(^{x} X_t)$, functional It\^o formulas are established when $X$ is a stable symmetric process. Singular cases when $x\mapsto (\nabla^w_x F_t)(^{x}X_t)$ is smooth off random bounded variation curves are also discussed. The results of this paper extend previous change of variable formulas in Cont and Fourni\'e~\cite{cont} and also Peskir~\cite{peskir}, Feng and Zhao~ \cite{feng} and Elworhty, Truman and Zhao~\cite{elworthy} in the context of path-dependent functionals. In particular, we provide a pathwise path-dependent version of the classical F\"{o}llmer-Protter-Shiryaev~\cite{follmer} formula for continuous semimartingales.
\end{abstract}

\maketitle

\section{Introduction}
The celebrated It\^o formula is the fundamental change of variables formula deeply connected with the concept of quadratic variation of semimartingales. It was initially conceived by Kiyosi It\^o and since then many authors have been extending his formula either relaxing smoothness of the transformation or generalizing to more general stochastic processes.

After It\^o, perhaps the major contribution towards a change of variables formula without $C^2$ assumption was due to the classical works by Tanaka, Wang and Meyer by making a beautiful use of the local time concept earlier introduced by Paul Levy. They proved that if $F:\mathbb{R}\rightarrow \mathbb{R}$ is convex then

$$F(B(t))= F(B(0)) + \int_0^t\nabla_-F(B(s))dB(s) + \frac{1}{2}\int_{-\infty}^{\infty}\ell^x(t)\rho(dx)$$
where $B$ is the Brownian motion, $\ell^x(t)$ is the correspondent local time two-parameter process at $(t,x)\in \mathbb{R}_+\times \mathbb{R}$ and $\rho$ is the Radon measure related to the generalized second-order derivative of $F$. A different extension to absolutely continuous functions with bounded derivatives is due to Bouleau and Yor~\cite{bouleau}

\begin{equation}\label{pe}
F(B(t))=F(B(0)) + \int_0^t\nabla F(B(s))dB(s) - \frac{1}{2}\int_{-\infty}^{\infty}\nabla F(x)d_x\ell^x(t)
\end{equation}
and later on extended by F\"{o}llmer, Protter and Shiryaev~\cite{follmer} and Eisenbaum~\cite{eisenbaum} to functions in the Sobolev space $\text{H}^{1,2}_{loc}(\mathbb{R})$ of generalized functions with weak derivatives in $L^2_{loc}(\mathbb{R})$. In this case, the correction term in~(\ref{pe}) is given by an $d_x\ell^x(t)$-integral in $L^2(\mathbb{P})$-sense where $\mathbb{P}$ is the Wiener measure. See also Bardina and Rovira~\cite{bardina} for the case of elliptic diffusions and Russo and Vallois \cite{russo} for the general semimartingale case composed with $C^1$ functions.

%We refer the reader to Ghomrasni and G. Peskir~\cite{gos} for a survey on extensions of It\^o formulas $f(X(t))$ where $f:\mathbb{R}\rightarrow \mathbb{R}$ is a measurable function and $X$ is a semimartingale.

Inspired by the two-dimensional Lebesgue-Stieltjes integration methodology of Elworthy, Truman and Zhao~\cite{elworthy}, a different pathwise argument was introduced by Feng and Zhao~\cite{feng,feng1} based on Young/Rough Path~(see~e.g~\cite{friz}) integration theory. They proved that the local time curves $x\mapsto \ell^x(t)$ of any continuous semimartingale $X$ admits $p$-variation $(p >2)$ almost surely for any $t \ge 0$. In this case, the pathwise rough path integral $\int_{-\infty}^{+\infty}\nabla_- F(x)d_x\ell^x(t)$ can be used as the correction term in the change of variable formula for $X$ as follows

$$F(X(t))-F(X(0)) = \int_0^t\nabla_-F(X(s))dX(s) -\frac{1}{2}\int_{-\infty}^{\infty}\nabla_- F(x)d_x\ell^x(t), 0\le t\le T,$$
where $F:\mathbb{R}\rightarrow \mathbb{R}$ is an absolutely continuous function with left-continuous left derivative $\nabla_-F$
with finite $p$-variation where $1\le p \le 3$.

One important class of semimartingale transformations which cannot be recovered by the previous methods is the following one

\begin{equation}\label{pathDF}
X_t\mapsto F_t(X_t); t\ge 0
\end{equation}
where $X_t=\{X(u); 0\le u\le t\}$ is the semimartingale path up to time $t$ and $F_t:C([0,t];\mathbb{R})\rightarrow\mathbb{R}; t\ge 0$ is a family of functionals defined on the space of real-valued continuous functions $C([0,t];\mathbb{R})$ on the intervals $[0,t]; t\ge 0$. Path-dependent transformations of type (\ref{pathDF}) have been studied in the context of the so-called functional stochastic calculus introduced by Dupire~\cite{dupire} and systematically studied by Cont and Fournie \cite{cont,cont1}. In fact, this approach has been recently studied by many authors in the context of path-dependent PDEs and path-dependent optimal stochastic control problems. We refer the reader to e.g~\cite{touzi1, touzi2, touzi3, LOS, cosso,cosso1, flandoli, buck, keller} for a detailed account on this literature. In this case, the usual space-time derivative operators are replaced by the so-called horizontal and vertical derivative operators, given by $\nabla^h F$ and $\nabla^v F$, respectively. Under suitable regularity conditions ($\mathbb{C}^{1,2}$ in the functional sense), one can show that if $X$ is a continuous semimartingale then

\begin{equation}\label{c12formula}
F_t(X_t) = F_0(X_0)+ \int_0^t\nabla^hF_s(X_s)ds + \int_0^t\nabla^vF_s(X_s)dX(s) + \frac{1}{2}\int_0^t\nabla^{v,2}F_s(X_s)d[X,X](s); t\ge 0
\end{equation}
where $\nabla^{v,2}F$ is the second order vertical derivative  and $[X,X]$ is the standard quadratic variation of $X$. See \cite{cont,dupire} for further details. Under weaker regularity assumptions, Le\~ao, Ohashi and Simas \cite{LOS} have extended (\ref{c12formula}) for functionals $F$ which do not admit second order vertical derivatives. By means of a weaker version of functional calculus, the authors show that path dependent functionals with rough regularity in the sense of $(p,q)$-variation are weakly differentiable and, in particular, they satisfy

%have established It\^o-type formulas for differentiable functionals of the form

\begin{equation}\label{pathFPS}
F_t(B_t) = F_0(B_0) + \int_0^t\mathcal{D}F_s(B_s)dB(s) + \int_0^t\mathcal{D}^{\mathcal{F},h}F_s(B_s)ds -\frac{1}{2}\int_0^t\int_{-\infty}^{+\infty} \partial_x F_s(^{x}B_s)d_{(s,x)}\ell^{x}(s),
\end{equation}
where the operators $(\mathcal{D}F,\mathcal{D}^{\mathcal{F},h}F)$ are similar in nature to $(\nabla^vF(B),\nabla^h F(B))$. The $d_{(s,x)}\ell^x(s)$-integral in (\ref{pathFPS}) is considered in the $(p,q)$-variation sense based on the pathwise 2D Young integral (see~\cite{young1}) where $\ell$ is the Brownian local-time. The integrand is a suitable space derivative of $F$ composed with a ``terminal value modification" $^{x}B_t$ defined by the following pathwise operation: For a given path $\eta_t:C([0,t];\mathbb{R})\rightarrow \mathbb{R}$, then

$$
^{x}\eta_t(u):=\left\{
\begin{array}{rl}
\eta(u); & \hbox{if}~0\le u<t \\
x;& \hbox{if}~u=t.
\end{array}
\right.
$$
In this work, our goal is to study a number of path-dependent It\^o formulas $F(X)$ beyond the smooth case of functionals with $\mathbb{C}^{1,2}$-regularity, where $X$ is an arbitrary semimartingale with continuous paths. Based on the framework of pathwise functional calculus, we establish a pathwise local-time decomposition

\begin{equation}\label{pathFPS1}
F_t(X_t) = F_0(X_0)+ \int_0^t\nabla^hF_s(X_s)ds + \int_0^t\nabla^wF_s(X_s)dX(s) - \frac{1}{2}\int_{-\infty}^{+\infty}\int_0^t(\nabla^w_xF_s)(^{x}X_s)d_{(s,x)}\ell^x(s)
\end{equation}
where $(\nabla^w_x F_s)(^{x}X_s)$
is a weak derivative of $F$ with respect to the terminal value $x$ of the modified path $^{x}X_s$
and $(\nabla^w F_s)(X_s) = (\nabla^w_x F_s)(^{x} X_s)|_{X(s)=x}$. The double integral is interpreted as a space-time 2D-Young integral with differential $d_{(s,x)}\ell^x(s)$ where $\ell$ is the local time of $X$. We study differential representations of form (\ref{pathFPS1}) under a set of assumptions related to rough variations in time and space: \textit{Two-parameter H\"{o}lder control}: For each $L>0$, there exists a constant $C$ such that

\begin{equation}\label{roughc}
|\Delta_i\Delta_j(\nabla^w_x F_{t_i})(^{x_j}c_{t_i})|\leq C|t_i-t_{i-1}|^{1/\tilde{p}}|x_j-x_{j-1}|^{1/\tilde{q}}
\end{equation}
for every partition $\{t_i\}_{i=0}^N\times \{x_j\}_{j=0}^{N^{'}}$ of $[0,T]\times [-L,L]$ and $c\in C([0,T];\mathbb{R})$. Here, $\Delta_j$ is the usual first difference operator and $\tilde{p},\tilde{q}\ge 1$ are constants such that

$$\alpha + \frac{1}{\tilde{p}}>1\quad \hbox{and}\quad\frac{(1-\alpha)}{2+\delta} + \frac{1}{\tilde{q}}>1$$
for some $\alpha \in (0,1)$ and $\delta >0$.
\textit{General}~$(a,b)$-\textit{variation}: In the particular case when $X$ is a continuous symmetric stable process with index $1<\beta\le 2 $, we establish formula (\ref{pathFPS1}) under general $(a,b)$-variation regularity assumption

\begin{equation}\label{gcase}
\sup_{\pi}\Bigg\{\Bigg[\sum_{j=1}^{N'}\Big[\sum_{i=1}^N|\Delta_i\Delta_j  (\nabla^w_x F_{t_i})(^{x_j}c_{t_i})|^a \Big]^{\frac{b}{a}} \Bigg]^{\frac{1}{b}} \Bigg\} < \infty;\quad c\in C([0,T];\mathbb{R})
\end{equation}
for $1\le a < \frac{2\beta}{\beta+1}$ and $1\le b < \frac{2}{3-\beta}$, where sup in (\ref{gcase}) is computed over the set of partitions $\pi$ of $[0,T]\times [-L,L]$ for each $L>0$. Other types of singularities are also discussed when $x\mapsto F_t(^{x}c_t)$ is smooth off path-dependent bounded variation curves.

The formulas presented in this article extend previous versions of path-dependent pathwise It\^o formulas given by Cont and Fournie~\cite{cont} and Dupire~\cite{dupire}. In relation to non-smooth path-dependent cases, we also extend Prop. 8.4 in Le\~ao, Ohashi and Simas~\cite{LOS} in the case when the path-dependent calculus is treated on the basis of functionals with a priori $(p,q)$-variation regularity rather than processes. In \cite{LOS}, the authors show that Wiener functionals with finite $(p,q)$-regularity of the form (\ref{roughc}) are weakly differentiable. In the present work, in the context of pathwise functional calculus, we show that this type of regularity also provides differential representations for path-dependent functionals driven by generic continuous semimartingales.

The level of regularity that we impose on the path-dependent functionals can be compared with the pioneering works of Elworthy, Truman and Zhao~\cite{elworthy}, Peskir \cite{peskir} and Feng and Zhao~\cite{feng,feng1} who obtain extensions of non-path dependent change of variables formulas by means of pathwise arguments based on Lebesgue-Stieltjes/Young/rough path type integrals. Our first result (Theorem~\ref{Ito4}) extends the classical result due to~\cite{peskir, elworthy} for functionals with singularity at path-dependent bounded variation curves. Applications to some path-dependent payoffs in Mathematical Finance are briefly discussed. The change of variable formulas under $(a,b)$-regularity (\ref{gcase}) (Proposition \ref{corPQ}) extend \cite{feng,feng1} with the restriction that the underlying noise is a continuous symmetric stable process. The general semimartingale case is treated in Theorem \ref{thYoung} under more restrictive assumptions on $\nabla^w F$ based on (\ref{roughc}),

One typical class of examples which fits into the assumptions of our theorems can be represented by

$$\int_{-\infty}^{X(t)}Z_t(X_t;y)dy$$
where $X$ is the semimartingale noise which induce the underlying filtration and $Z = \{Z(\cdot; x):C([0,t];\mathbb{R})\rightarrow \mathbb{R}; (t,x)\in[0,T]\times \mathbb{R}\}$ is a family of functionals satisfying some two-parameter variation regularity of the forms (\ref{gcase}) or (\ref{roughc}). This can be seen as a pathwise path-dependent version of the classical F\"{o}llmer-Protter-Shiryaev formula (see \cite{follmer}) for continuous semimartingales.

This paper is organized as follows. Section \ref{molsec} presents basic notations and some preliminary results. In Section \ref{smoothcase}, we investigate It\^o formulas for path-dependent functionals which are regular off path-dependent bounded variation curves. Applications to some running maximum/minimum functionals arising in Mathematical Finance are presented. Section \ref{bivariationsection} presents It\^o formulas under $(p,q)$-variation assumption of the particular form (\ref{roughc}). Section \ref{jointsection} treats the general case (\ref{gcase}) under the assumption that the underlying driving noise is a symmetric stable process.

\section{ Functional Mollification}\label{molsec}

Throughout this paper we are given a stochastic basis $(\Omega, \mathbb{F}, \mathcal{F}, \mathbb{P})$. Here, the set $\Omega:=\{\omega\in C([0,+\infty);\mathbb{R}); \omega(0) = z\}$ is the set of real-valued continuous paths on $\mathbb{R}_+$ which starts at a given $z\in \mathbb{R}$, $X$ is the canonical process, $\mathbb{F} := (\mathcal{F}_t)_{t\ge 0}$ is the natural filtration generated by $X$, $\mathcal{F}$ is a sigma-algebra such that $\mathcal{F}_t \subset \mathcal{F}~\forall t\ge 0$ and $\mathbb{P}$ is the semimartingale measure on $\Omega$. The usual quadratic variation will be denoted by $[X,X]$ and we recall the local time of $X$ is the unique random field $\{\ell^x(t); (x,t)\in \mathbb{R}\times \mathbb{R}_+\}$ which realizes

$$\int_0^t f(X(s))d[X,X](s) = \int_{\mathbb{R}}\ell^x(t)f(x)dx; t \ge 0$$
for every bounded Borel measurable function $f:\mathbb{R}\rightarrow \mathbb{R}$. Frequently, localization procedures will be necessary to handle the path-dependence. For this reason, for a given $M> 0$, we set

$$T_M:=\inf\{t\ge 0; |X(t)|> M\}\wedge T$$
where $0 < T < + \infty$ is a fixed terminal time and $a\wedge b:=min\{a,b\}$. The stopped semimartingale will be denoted by $X^M(t):=X(T_M\wedge t); 0\le t\le T$. We denote $D([0,t];\mathbb{R})$ ($C([0,t];\mathbb{R})$) as the linear space of $\mathbb{R}$-valued c\`adl\`ag (continuous) paths on $[0,t]$ and we set $\Lambda:=\cup_{0\le t\le T}D([0,t];\mathbb{R})$ and $\hat{\Lambda}:=\cup_{0\le t\le T}C([0,t];\mathbb{R})$. In order to make clear the information encoded by a path $x \in D([0,t];\mathbb{R})$ up to a given time $0\le r\le t$, we denote $x_r:= \{x(s): 0 \le s \le r \}$ and the value of $x$ at time $0 \le u \le t$ is denoted by $x(u)$. This notation is naturally extended to processes. Throughout this paper, if $f$ is a real-valued function defined on a metric space $E$, then

$$\Delta_jf(x_j):=f(x_j) - f(x_{j-1})$$
for every sequence $\{x_j\}_{j=0}^m\subset E$. In particular, if $\varphi:[0,T]\times \mathbb{R}\rightarrow \mathbb{R}$ then 

$$\Delta_j \Delta_i \varphi(t_i,x_j):= \varphi(t_i,x_j) - \varphi(t_{i-1},x_j)- \big( \varphi(t_i, x_{j-1}) - \varphi(t_{i-1},x_{j-1})   \big)$$
for any sequence $\{t_{i}\}_{i=0}^m\times \{x_k\}_{k=0}^p\subset[0,T]\times \mathbb{R}$. 

For reader's convenience, let us recall some basic objects of the pathwise functional calculus. We refer the reader to Dupire~\cite{dupire} and Cont and Fournie~\cite{cont,cont1} for further details. Throughout this article, if $w\in \Lambda$, then for a given $\gamma>0$ and $h\in\mathbb{R}$, we denote

%\label{hor}
$$
w_{t,\gamma}(u) := w(u);0\le u\le t \quad \text{and}~w_{t,\gamma}(u):=w(t); t < u \le t+\gamma,
$$
%\label{ver}

$$
w^h_{t}(u) :=w(u);0\le u< t\quad \text{and}~w^h_{t}(t):= w(t) + h.
$$
If $x\in\mathbb{R}$, we denote

$$
^{x}w_t(u) :=w(u);0\le u< t\quad \text{and}~^{x}w_t(u):= x; u=t.
$$
A natural metric on $\Lambda$ is given by

$$d_{\infty}\big((t,w); (s,v)   \big):= |t-s| + \sup_{0\le u\le T}|w_{t,T-t}(u) - v_{s,T-s}(u)|;$$
for $(w, v)$ in $\Lambda\times \Lambda$. Throughout this article, a functional $F = \{F_t; 0\le t\le T\}$ is just a family of mappings $F_t:D([0,t];\mathbb{R})\rightarrow \mathbb{R}$ indexed by $t\in[0,T]$. In the sequel, continuity of functionals is defined as follows (see e.g~\cite{cont}):

\begin{definition}
A functional $F = \{F_t; 0\le t\le T\}$ is said to be $\Lambda$-continuous at $c\in \Lambda$ if $\forall \varepsilon >0$, there exists $\delta(c,\varepsilon)=\delta>0$ such that if $\eta\in \Lambda$ satisfies $d_{\infty}((t,c);(\eta,y))< \delta$ then $|F_t(c) - F_{s}(\eta)|<\varepsilon $. We say that $F$ is $\Lambda$-continuous if it is continuous for each $c\in\Lambda$.

%Moreover, we say that $F$ is $\Lambda$-uniformly continuous if for every $\varepsilon > 0$ there exists $\delta>0$ such that $|F_t(x) - F_{s}(y)|<\varepsilon %$ whenever $d_{\infty}((t,x); (s,y)) < \delta$.
\end{definition}

We recall the vertical derivative of a functional $F\in \Lambda$ is defined as

\begin{equation}\label{verdef}
\nabla^{v}F_t(c_t):= \lim_{h\rightarrow 0}\frac{F_{t}(c^h_t) - F_t(c_t)}{h}
\end{equation}
whenever the right-hand side of (\ref{verdef}) exists for every $c\in \Lambda$. We define $\nabla^{v, (2)}F:= \nabla^{v} (\nabla^v F)$ whenever this operation exists. The horizontal derivative is defined by the following limit

\begin{equation}\label{hordef}
\nabla^{h}F_t(c_t):= \lim_{\gamma\rightarrow 0^+}\frac{F_{t+\gamma}(c_{t,\gamma}) - F_t(c_t)}{\gamma}
\end{equation}
whenever the right-hand side of (\ref{hordef}) exists for every $c\in \Lambda$.

An $\mathbb{F}$-adapted continuous process $Y$ may be represented by the identity

\begin{equation}\label{repfunc}
Y(t) = \hat{F}_t(X_t); 0\le t\le T,
\end{equation}
where $\hat{F} =\{\hat{F}_t; 0\le t\le T\}$ is a family of functionals $\hat{F}_t:C([0,t];\mathbb{R})\rightarrow \mathbb{R}$ representing the dependence of $Y$.

Since $Y$ is non-anticipative, $Y(\omega,t)$ only depends on the restriction of $\omega$ over $[0,t]$. In order to perform the standard pathwise functional calculus in the sense of Dupire~\cite{dupire} and Cont and Fournie~\cite{cont}, one has to assume there exists a family of functionals $F = \{F_t; 0\le t\le T\}$ defined on $\Lambda$ which is consistent to $\hat{F}$ in the sense that

$$F_t(c_t) = \hat{F}_t(c_t)\quad \forall c\in \hat{\Lambda}.$$
Indeed, the concept of vertical derivative forces us to assume this. Throughout this article, whenever we write $Y = F(X)$ for $F$ defined on $\Lambda$, it is implicitly assumed that $F$ is a consistent extension of a functional representation $\hat{F}$ which realizes (\ref{repfunc}). This motivates the following definition.

\begin{definition}
A non-anticipative functional is a family of functionals $F = \{F_t; 0\le t\le T\}$ where

$$F_t:D([0,t]; \mathbb{R})\rightarrow \mathbb{R}; c\mapsto F_t(c_t)$$
is measurable w.r.t the canonical filtration $\mathcal{B}_t$ in $D([0,t];\mathbb{R})$ for each $t\in [0,T]$.
\end{definition}

In the sequel, let $\mathbb{C}^{1,2}$ be the space of functionals $F$ which are $\Lambda$-continuous and it has $\Lambda$-continuous derivatives $\nabla^h F, \nabla^{v,(i)}F$ for $i=1,2 $. The above notion of continuity is enough to apply the standard functional stochastic calculus techniques in the smooth case $F\in \mathbb{C}^{1,2}$. However, in order to employ mollification techniques to treat non-smooth dependence (in the sense of differentiation) of $F$ w.r.t $X$, we need the following notion of continuity.

\begin{definition}
We say that a family of functionals $\{H^x:\Lambda \rightarrow \mathbb{R}; x\in \mathbb{R}\}$ is state-dependent $\Lambda$-continuous
at $v\in \Lambda$ if there exists $\phi\in L^1_{loc}(\mathbb{R})$ such that for every $\varepsilon >0$, there exists $\delta>0$ such that

$$d_{\infty}\big((t{'}, c); (t, v)\big) < \delta\Longrightarrow |H^x_{t{'}}(c_{t'})  - H^x_t(v_{t})|\le \varepsilon \phi(x); \forall x\in \mathbb{R}.$$
When the family $\{H^x; x\in \mathbb{R}\}$ is state-dependent $\Lambda$-continuous for every $v\in \Lambda$, we say that it is state-dependent $\Lambda$-continuous.
\end{definition}

\begin{remark}
If $\{H^x; x\in\mathbb{R}\}$ is state-dependent $\Lambda$-continuous,then it is $\Lambda$-continuous for each $x\in \mathbb{R}$.
\end{remark}

\noindent \textbf{Example:} Let us give an example of a state-dependent $\Lambda$-continuous family of functionals. In the sequel, $(x)^+:= max\{x;0\}, x\in \mathbb{R}$. For a given constant $K$, we consider $F_t(c_t) = \Big(\underset{ 0\le s\le t}{\operatorname{sup}}c(s)-K\Big)^+$. Then, $F_t(^{x}c_t) = \Big( \underset{ 0\le s \le t}{\operatorname{sup}}^{x}c(s) - K \Big)^+$ for each $x\in \mathbb{R}, c\in \Lambda$ and we readily see that the family $c\mapsto F(^{x}c); x\in \mathbb{R}$ is state-dependent $\Lambda$-continuous.

\

For the remainder of this paper it will be convenient to use the following notation: For a given family of functionals $F = \{F_t; 0\le t\le T\}$, we define

\aaa{
\label{new-f}
\mathcal{F}^x(c_t):=F_t(^{x}c_t)
}
for $c\in \Lambda$ and $x\in \mathbb{R}$. This notation will be useful to compute horizontal derivatives from a state-dependent $\Lambda$-continuous family of the form $\{\mathcal{F}^x; x\in \mathbb{R}\}$.

The strategy to get functional It\^o formulas under non-smooth conditions will be based on path-dependent mollification techniques on the state of the functional. Indeed, in this article we are only interested in relaxing vertical smoothness of path-dependent functionals. In this case, it will be sufficient for us to deal with one parameter mollification.

For a given non-negative smooth function $\rho\in C^{\infty}_c(\mathbb{R})$ such that $supp~\rho\subset (0,2)$, $\int_{\mathbb{R}}\rho(x)ds=1$, we set $\rho_n(x):= n\rho(nx); x\in \mathbb{R}; n\ge 1.$ If $x\mapsto \mathcal{F}^x_t(c_t) \in L^1_{loc}(\mathbb{R})$ for every $c\in\Lambda$, then we define

\aaa{
\label{mol}
F^n_t(c_t; x):=  \big(\rho_n\star \mathcal{F}^\cdot_t(c_t) \big)(x); x\in \mathbb{R}, c\in \Lambda, t\in [0,T],
}
where $\star$ denotes the usual convolution operation on the real line. From this convolution operator, we define the following non-anticipative functional

$$F^n_t(c_t) := \int_{\mathbb{R}}\rho_n(c(t)-y)\mathcal{F}^y_t(c_t)dy; 0\le t\le T.$$
One should notice that $F^n_t(^{x}c_t) = F^n(c_t; x); c\in \Lambda,~x\in\mathbb{R}$. In the sequel, we need a notion of boundedness to treat path-dependent functionals.

\begin{definition} \label{lbp}
We say that a family of functionals $F = \{F_t; 0\le t\le T\}$ is boundedness-preserving if for every compact subset $K$ of $\mathbb{R}$, there  exist $C_K > 0$ such that $|F_\cdot(c_\cdot)| \le C_K$ for every $c_\cdot \in D([0,\cdot]; K)$. A family of functionals $H^x:\Lambda \rightarrow \mathbb{R}; x\in \mathbb{R}$ is state boundedness-preserving if for every compact sets $K_1, K_2\subset \mathbb{R}$, there exists a constant $C_{K_1,K_2}>0$ such that

$$|H^x_\cdot(c_{\cdot})|\le C_{K_1,K_2}\quad \forall c\in D([0,\cdot]; K_1)~\text{and}~\forall x\in K_2.$$
\end{definition}

Let us now introduce the following hypotheses

%\noindent \textbf{Assumption A1:} $x\mapsto \Im_t (c_t; x) \in L^1_{loc}(\mathbb{R})$ for every $c\in\Lambda$ and $t\in [0,T]$.

\

\noindent \textbf{Assumption A1:}

\noindent (i) The family of functionals $\{\mathcal{F}^y;y\in \mathbb{R}\}$ is state-dependent $\Lambda$-continuous and state-boundedness-preserving.

\noindent (ii) $x\mapsto  \mathcal{F}^x_t(c_t)$ is a continuous map for every $c\in \Lambda$ and $t\in [0,T]$.

\noindent (iii) $x\mapsto  \mathcal{F}^x_t(c_t)$ has weak derivative for every $c\in \Lambda$ and $t\in [0,T]$.

\

\noindent \textbf{Assumption A2:} For each $y\in \mathbb{R}$, $\mathcal{F}^y$ has horizontal derivative $\nabla^h\mathcal{F}^y(c)~\forall c\in \Lambda$. Moreover, the family $\{\nabla^h\mathcal{F}^y; y\in\mathbb{R}\}$ is state boundedness-preserving. The map $y\mapsto \nabla^h \mathcal{F}^y_t(c_t)$ is continuous for every $c\in \Lambda$. The family of functionals $\{\nabla^h\mathcal{F}^y; y\in \mathbb{R}\}$ is state-dependent $\Lambda$-continuous.

\

Throughout this paper, the weak derivative of $x\mapsto \mathcal{F}^x_t(c_t)$ will be denoted by $(\nabla^w_x F_t)(^{x}c_t)$ and we set

$$\nabla^w F_t(c_t):=(\nabla^w_x F_t)(^{x}c_t)|_{x=c(t)}; c\in \Lambda.$$
Of course, $(\nabla^w_x F_t)(^{\cdot}c_t)\in L^1_{loc}(\mathbb{R})$ is uniquely specified by the property

$$\int_{\mathbb{R}}\mathcal{F}^x_t(c_t)\varphi'(x)dx = -\int_{\mathbb{R}}(\nabla^w_x F_t)(^{x}c_t)\varphi(x)dx;~c\in \Lambda,$$
for every real-valued smooth function $\varphi \in C^1_{c}(\mathbb{R})$.

\

If Assumptions A1.(iii) holds, then $F^n_t(^{\cdot}c_t)\in C^\infty(\mathbb{R})~\forall c\in \Lambda, t\in [0,T], n\ge 1$ and integration by parts yields

$$\nabla_x F^n_t(^{x}c_t) = \int_{\mathbb{R}}\rho_n(x-y)(\nabla^w_yF_t)(^{y}c_t)dy.$$
Moreover, the vertical derivative of functional mollification is given by

$$\nabla^{v,i}F^n_t(c_t) =\nabla^{i}_x F^n_t(^{x}c_t)|_{x=c(t)}$$
for $i=1, 2$. To compute the horizontal derivative of mollifiers, the following simple lemma will be useful.

\begin{lemma}
\label{lem1}
Let $\mathcal{A}$ be a parameter set, and let $f: \mathcal A\times \mathbb{R} \to \Rnu$ be a function, continuous on the second variable and such that for each $a\in \mathcal{A}$, there exists the right derivative $\nab^+_x f(a,x); x\in \mathbb{R}$ which is bounded on $\mathcal A\times \mathbb{R}$. Suppose for each $(a,x)\in \mathcal{A}\times\mathbb{R}$, there exists $a_x \in \mathcal A$ such that $f(a,x+h) = f(a_x,h)$. Then, the ratio $\frac{f(a,x+h) - f(a,x)}{h}$ is bounded over $\mathcal{A}\times \mathbb{R}\times\mathbb{R}_+$. The analogous result also holds for the ratio $\frac{f(a,x-h) - f(a,x)}{-h}$ under boundedness condition on $\nabla^-f(a,x)$ over $\mathcal{A}\times \mathbb{R}$.

%For each fixed $a\in P$, let $f_n(a,h)$ be the  mollifier for $f(a,h)$ with respect to $h$. Then,
%$\pl_h f_n(a,h)\big|_{h=0}$ is the mollifier for $\pl^+_h f(a,h)\big|_{h=0}$.
\end{lemma}
\begin{proof}
Let us fix an arbitrary pair $(a,x)\in \mathcal{A}\times \mathbb{R}$ and we define the set
\aa{
H(a,x) = \{h\ge 0: |f(a,x+h) - f(a,x)| \lt Ch\},
}
where $C = 1 + \sup_{a\in \mathcal A, x\in \mathbb{R}}|\nab^+_x f(a,x)|$. The set $H(a,x)$ is closed and it contains a closed interval $[0,l_{a,x}]$. Let $L_{a,x}$ be the length of the maximal interval of the form $[0,l_{a,x}]$
contained in $H(a,x)$. Suppose $L_{a,x} < \infty$.
Take $h = L_{a,x} + k$, where $k\in H(a_{L_{a,x}},x)$, and $a_{L_{a,x}}$ is such that
$f(a,L_{a,x} + y) = f(a_{L_{a,x}}, y)$, $y\in\Rnu$. We have
\mm{
|f(a,x+h) - f(a,x)| \le |f(a,L_{a,x}+x+k) - f(a,L_{a,x}+x)|  + |f(a,L_{a,x}+x) - f(a,x)| \\
= |f(a_{L_{a,x}},x+k) - f(a_{L_{a,x}},x)| + |f(a,x+L_{a,x}) - f(a,x)|
\lt Ck + CL_{a,x} = Ch.
}
Thus, $L_{a,x}$ is not maximal and we have a contradiction. This implies that $L_{a,x}=\infty$ and therefore, the ratio
$\frac{f(a,x+h) - f(a,x)}{h}$ is bounded on $\mathcal A\times \mathbb{R}\times\mathbb{R}_+$. The proof for the ratio related to the left-derivative is obviously the same.
\end{proof}

\begin{remark}
\label{rem1}
\rm
Lemma \ref{lem1} also holds for functions $f: \mathcal A\times [A,B] \to \Rnu$, where $[A,B]\subset \Rnu$.
Indeed, extend the function $f$ to the whole real line by setting $f(a,x) = f(a,A)$ for $x<A$ and $f(a,x) = f(a,B)$ for $x>B$.
\end{remark}

\begin{remark}
\label{rem2}
\rm
Note that the horizontal derivative $\nab^h F_t(c_t)$ can be regarded as
the right derivative $\nab^+_\gm F_{t+\gm}(c_{t,\gm})$ of the function
$\gm\mto F_{t+\gm}(c_{t,\gm})$ at point $\gamma=0$, where the pair $(t,c) \in [0,T]\times \Lambda$ is interpreted
as a parameter. The assumption $f(a,x+h) = f(a_x,h)$ in Lemma \ref{lem1} is interpreted in the setup of functional calculus as follows:

$$F_{t+\gm+h}(c_{t,\gm+h}) = F_{(t+\gm)+h}(\td c_{t+\gm,h}),$$
where $\td c_{t+\gm} = c_{t,\gm}.$ The same remark holds for the functional $\mathcal{F}^y_t(c_t)$ with parameters $(t,c,y)\in [0,T]\times \Lambda \times U$ in some bounded open subset $U\subset\mathbb{R}$.
\end{remark}

\begin{lemma}\label{horderm}
Assume that for each $y\in \mathbb{R}$, $\mathcal{F}^y$ is $\Lambda$-continuous, $\mathcal{F}^y$ has horizontal derivative and the family $\{\nabla^h\mathcal{F}^y; y\in\mathbb{R}\}$ is state boundedness-preserving. Then, for each $n\ge 1$, $t\in [0,T]$,
and $c\in \Lambda$ taking values in a compact subset of $\mathbb{R}$, we have

\begin{equation}\label{hnder}
\nabla^h F^n_t(c_t) = \int_{\mathbb{R}}\rho_n(c(t)-y)\nabla^h \mathcal{F}^y_t(c_t)dy = \big(\rho_n\star \nabla^h \mathcal{F}^\cdot_t(c_t)\big)(c(t))
\end{equation}
\end{lemma}
\begin{proof}
We fix $t\in [0,T]$ and a path $c\in \Lambda$ over $[0,t]$ such that $c(u)\in K; 0\le u\le t$, where $K$ is a compact set. We also fix $n\ge 1$. Indeed, by the very definition
\aaa{
\label{f12}
\frac{F^n_{t+\gm}(c_{t,\gm}) - F^n_t(c_t)}{\gm} = \int_{c(t)-\frac2{n}}^{c(t)} \rho_n(c(t)-y) \, \Bigg[\frac{\mathcal{F}^y_{t+\gm}(c_{t,\gm}) - \mathcal{F}^y_t(c_t)}{\gm}\Bigg]\, dy
}
for $\gamma>0$.
We claim that the ratio $\frac{\mathcal{F}^y_{t+\gm}(c_{t,\gm}) - \mathcal{F}^y_t(c_t)}{\gm}$ is bounded over $(\gamma, y)\in [0,T-t]\times [c(t)-\frac{2}{n},c(t)]$.
Indeed, we shall apply Lemma \ref{lem1} to the function $\gm\mto \mathcal{F}^y_{t+\gm}(c_{t,\gm})$ defined on $[0,T-t]$ regarding
$y\in  [c(t)-\frac{2}{n},c(t)]$
as a parameter (see Remark \ref{rem2}). From the $\Lambda$-continuity of $\mathcal{F}^y$, one can easily check that
$$\gamma\mapsto\mathcal{F}^y_{t+\gm}(c_{t,\gm})$$
is continuous over $[0,T-t]$. Extend the function $\gamma\mapsto \mathcal{F}^y_{t+\gamma}(c_{t,\gamma})$ to $\Rnu$  by the constant values that it attains at the end points of $[0,T-t]$. As we already mentioned in Remark \ref{rem2}, for each $y\in \mathbb{R}$, the right derivative $\nabla^+_\gm \mathcal{F}^y_{t+\gm}(c_{t,\gm})$ at $\gamma_0$ is the horizontal derivative $\nabla^h\mathcal{F}^y_{t+\gamma_0}(c_{t,\gamma_0})$ for $\gamma_0\in [0,T-t]$. By the state boundedness-preserving assumption, $\nabla^+_\gm \mathcal{F}^y_{t+\gm}(c_{t,\gm})$ is bounded over $[0,T-t]\times [c(t)-\frac{2}{n},c(t)]$. Again, taking into account Remark \ref{rem2}, we conclude that we are in the situation of Lemma \ref{lem1}. Bounded convergence theorem allows us to take the limit into the integral sign in \eqref{f12} as $\gm\to 0$ which provides (\ref{hnder}). This completes the proof.
\end{proof}

\begin{lemma}\label{preIto}
If $F$ is a non-anticipative functional satisfying Assumptions A1(i) and A2, then for each positive integer $n\ge 1$, we have
\aaa{
\label{preIto1}
F^n_t(X_t) = F^n_0(X_0) + \int_0^t\nabla^h F^n_s(X_s)ds + \int_0^t\nabla^v F^n_s(X_s)dX(s)
+ \frac{1}{2}\int_0^t\nabla^{v,2}F^n_s(X_s)d[X,X](s)~\text{a.s.}
}
for $0\le t\le T.$
\end{lemma}
%For a given $M> 0$, let $T_N:=\inf\{t\in [0,T]; |X(t)|> M\}$. Since $X$ is a semimartingale, then $T_M< \infty~a.s$ for every $M> 0$. Let us fix $M> 0$ and we set $\Lambda_M:=\cup_{0\le t\le T}D([0,t]; [-M,M])$. Since we are working on the compact set $[-M,M]$, then we shall assume that $x\mapsto \Im_t(c_t; x)$ vanish outside $[-M,M]$.
\begin{proof}
%Let us fix $M>0$ and a positive integer $n\ge 1$. At first, since we are working with the stopped semimartingale paths $X^M_\cdot$, then we shall restrict the analysis to paths $c\in \Lambda$ with support on $[-M,M]$.
Let us fix $n\ge 1$. By routine stopping arguments, we may assume that $X$ is bounded. Hence, we shall assume that all paths $c\in \Lambda$ take values on a common compact subset of $\mathbb{R}$. First we show that $F^n$ is $\Lambda$-continuous. Indeed, by the very definition
\aa{
F^n_t (c_t) =\int_{-\infty}^{\infty} \rho_n(c(t)-y)  \mathcal{F}^{y}_t(c_t) dy.
}
Let us fix an arbitrary $c\in \Lambda$. The $\Lambda$-continuity of $F^n$ follows immediately from the state-dependent continuity of $\{\mathcal{F}^y; y\in \mathbb{R}\}$ and the triangle inequality:
\begin{eqnarray*}
|F^n_t(c_t) - F^n_{t{'}}(w_{t'})|&\le& \int_{K}\big|\rho_n \big(c(t)-y\big) - \rho_n \big(w(t{'})-y\big)\big||\mathcal{F}^y_t(c_t)|dy\\
& &\\
&+& \int_{K}\big|\rho_n\big(w(t{'})-y\big)\big||\mathcal{F}^y_t(c_t) - \mathcal{F}^y_{t'}(w_{t'})|dy.
\end{eqnarray*}
for $w\in \Lambda$, where $K$ is a compact set. By the very definition,

$$\nabla^{v,i}F^n_t(c_t) = n^{i+1}\int_{\mathbb{R}}\rho^{(i)}\big(n(c(t)-y)   \big)\mathcal{F}^y_t(c_t)dy; 0\le t\le T$$
for $i=1, 2$. Similarly, the $\Lambda$-continuity of $\nabla^{v,i}F^n$ follows immediately from the state-dependent continuity of $\{\mathcal{F}^y; y\in \mathbb{R}\}$ and the triangle inequality:

\begin{eqnarray*}
|\nabla^{v,i}F^n_t(c_t) - \nabla^{v,i}F^n_{t{'}}(w_{t'})|&\le& n^{i+1}\int_{K}\Big|\rho^{(i)}\big(n(c(t)-y)\big) - \rho^{(i)}\big(n(w(t{'})-y)\Big||\mathcal{F}^y_t(c_t)|dy\\
& &\\
&+& n^{i+1}\int_{K}\big|\rho^{(i)}\big(n(w(t')-y)\big)\big||\mathcal{F}^y_t(c_t) - \mathcal{F}^y_{t{'}}(w_{t'})|dy.
\end{eqnarray*}
By Lemma \ref{horderm} and triangle inequality,

\begin{eqnarray}
\label{hcont}|\nabla^{h}F^n_t(c_t) - \nabla^{h}F^n_{t'}(w_{t'})|&\le& \int_{K}\big|\rho_n \big(c(t)-y\big) - \rho_n \big(w(t')-y\big)\big||\nabla^h\mathcal{F}^y_t(c_t)|dy\\
\nonumber& &\\
\nonumber&+& \int_{K}\big|\rho_n\big(w(t')-y\big)\big||\nabla^h\mathcal{F}^y_t(c_t) - \nabla^h\mathcal{F}^y_{t'}(w_{t'})|dy.
\end{eqnarray}
Estimate (\ref{hcont}), the local integrability of $y\mapsto \nabla^h\mathcal{F}^y_t(c_t)$ and the state-dependent $\Lambda$-continuity of $\{\nabla\mathcal{F}^y; y\in \mathbb{R}\}$ yield the $\Lambda$-continuity of $\nabla^{h}F^n$.
%The $\Lambda$-vertically continuity of $\nabla^h F$ and the fact that $\nabla^hF_t (^{\cdot}c_t)\in L^1_{loc}$ yield $\nabla^h F^n$ is $\Lambda$-continuous.
Hence, $F^n$ is $\mathbb{C}^{1,2}$. The functional It\^o formula~(see e.g~\cite{dupire},~\cite{cont}) applied to the semimartingale $X$ yields

$$F^n_t(X_t) = F^n_0(X_0) + \int_0^{t}\nabla^hF^n_s(X_s)ds + \int_0^{t}\nabla^v F^n_s(X_s)dX(s) + \frac{1}{2}\int_0^{t}\nabla^{v,2}F^n_s(X_s)d[X,X](s)$$
for $0\le t\le T.$
\end{proof}

%\begin{remark}
%The running maximum $F_t(c_t):=\sup_{0\le s\le t}c(s); 0\le t\le T$ satisfies A1,A2,A3,A4.
%\end{remark}

\section{Path-dependent It\^o formula with singularity at random curves}\label{smoothcase}
In this section, we will investigate a path-dependent It\^o formula when the function $x\mapsto (\nabla^w_x F_t)(^{x} c_t)$ is smooth off path-dependent continuous bounded variation curves. The typical examples we have in mind are non-smooth functionals of the running maximum/minimum found in path-dependent payoffs arising in Mathematical Finance. Obtaining this type of It\^o's formula was inspired by Elworthy, Truman and Zhao~\cite{elworthy} who derived (non-path dependent) It\^o formulas where singularities are encoded by deterministic bounded variation curves. See also Peskir \cite{peskir}. At first, we remark that the classical occupation time formula also holds with path-dependent functions. We omit details of the proof which can be easily checked by well-known arguments.

\begin{lemma}\label{occupationformula}
Let $X$ be a continuous semimartingale with the local time $\{\ell^x(t); x\in \mathbb{R}, t\ge 0\}$. If $h:\Omega\times [0,T]\times \mathbb{R}\rightarrow \mathbb{R}$ is bounded and measurable, then for each $\omega\in \Omega$, we have
$$\int_0^t h(s,\om,X(s,\om)) d[X,X](s,\om) = \int_{-\infty}^\infty da \int_0^t h(s,\om,a) d_s\ell^a(s); 0\le t\le T.$$
\end{lemma}

%\begin{proof}
%We know that for each fixed level $a\in\mathbb{R}$, $\int_0^\infty 1\!\!1_{\mathbb{R}-a}(X(t,\omega))d_t\ell^a(t,\omega)=0~a.a~\omega\in\Omega$. Occupation formula (see e.g page 224 in \cite{karatsaz}) yields
%\mm{
%\int_0^t h(s,\om,X(s,\om)) d[X,X](s,\om) = \int_{-\infty}^\infty da \int_0^t h(s,\om,X(s,\om)) d_s\ell^a(s) \\
%= \int_{-\infty}^\infty da \int_{\{s\in [0,t]: X(s,\om) = a\}} h(s,\om,a) d_s\ell^a(s)
%= \int_{-\infty}^\infty da \int_0^t h(s,\om,a) d_s\ell^a(s).
%}
%and this completes the proof.

%-------------------------------------------------------------------------------------
%Let $\Lambda_t^a$ be the measure generated by the local time $\ell_t^a$ at a fixed level $a$. For a given $\omega\in \Omega$, we recall

%\aa{
%\Lambda_t^a\{s\in [0,t]: X(t,\om) \ne a\} = 0~0\le t\le T.
%}
%\end{proof}

Let $\gamma = \{\gamma_t; 0\le t\le T\}$ be a family of non-anticipative functionals such that for each $c\in C([0,T];\mathbb{R})$, $t\mapsto \gamma_t(c_t)$ is a continuous bounded variation path. In the sequel, to keep notation simple, for a given $M> 0$,
we set $\mathcal{C}_M:= C([0,T]; [-M,M])$ and

$$\mathcal{G}_M: = \{(t,x,c)\in [0,T]\times [-M,M]\times \mathcal{C}_M \}.$$
$$\mathcal{G}^{\gamma}_M: = \{(t,x,c)\in \mathcal{G}_M; -M< x < \gamma_t(c_t)~\text{or}~\gamma_t(c_t) <x < M\}.$$
$$\Gamma_{c,t}:= (-\infty, \gamma_t(c_t)) \cup (\gamma_t(c_t), +\infty);~c\in \mathcal{C}_M,~0\le t \le T.$$

%In the sequel, we denote $\gamma_Y(t):=\gamma_t(Y_t); 0\le t\le T$ if $Y$ is an $\mathbb{F}$-semimartingale.

Throughout this section, for a given $c\in\mathcal{C}_M$, we write $\nabla_x F_t(^{x}c_t)$ and $\nabla_x^{-} F_t(^{x}c_t)$ to denote the usual pointwise derivative
and left derivative, with respect to $x$, respectively. The second left derivative will be denoted by $\nabla^{-,2}_x$. Since $\gamma$ is non-anticipative, then $\gamma(X)$ is an adapted bounded variation process.

\begin{theorem}
\label{Ito4}
Let us assume that A1.(i,ii) and A2 hold and for each $t\in [0,T]$, the function $x\mto F_t(^x c_t)$ is $C^1$ on sets $\Gamma_{c,t}$ for $c\in C([0,T];\mathbb{R})$, where $\nabla_x F_t(^x c_t)$ is bounded on the set $\mathcal{G}^\gamma_M$ for every $M>0$.
We also assume that for each $t\in [0,T]$, there exist left and right limits of $\nab_x F_t(^x c_t)$ %at $\gm_c(t)$.
as $x\to  \gm_t(c_t) \pm$.
Furthermore, we assume that for any $t\in [0,T]$ and $c_t\in C([0,t],\Rnu)$, there exists the second left derivative
$\nabla^{-,2}_x  F_t(^x c_t)$ on $\Gamma_{c,t}$ which is bounded on $\Gamma_{c,t}\cap (-M,M)\times \mathcal{C}_M$ for every $M>0$. Moreover, $\nabla^{-,2}_x  F_t(^x c_t)$ has the left limit at $\gamma_t(c_t)$ for each $c\in \mathcal{C}_M$. %as $x\to  \gm_c(t)-$.
%and left continuous on the set $\Gamma_{c,t}$.
%$\mathcal{G}^\gamma_M$.
Finally, we assume that for any $c\in C([0,T];\mathbb{R})$, $\nab_x F_t(^{\gamma_t(c_t)-}c_t) - \nab_x F_t(^{\gamma_t(c_t)+}c_t)$ is continuous in $t$. If $X$ is a square-integrable continuous semimartingale, then

\mmm{
\label{formula5}
F_t(X_t) = F_0(X_0)+ \int_0^t\nabla^hF_s(X_s)ds + \int_0^t\nab^-_xF_s(^{X(s)}X_s)dX(s) \\
+\frac{1}{2}\int_0^t \nabla^{-,2}_x  F_s(^{X(s)-}X_s)d[X,X](s)
+ \frac{1}{2}\int_0^t(\nab_x F_s(^{\gm_s(X_s)+}X_s) - \nab_x F_s(^{\gm_s(X_s)-}X_s))d_s\td \ell^{\,0}(s) \; \text{a.s.}
}
for $0\le t\le T$, where $\{\td \ell^x(s); (x,s)\in \mathbb{R}\times \mathbb{R}_+\}$ is the local time of the semimartingale $\td X := X - \gm (X)$.
\end{theorem}

\begin{proof}
The proof uses some of the ideas from Corollary 2.1 of  Theorem 2.3 in \cite{elworthy}.
 At first, we prove the result for the stopped process $X^M$ where $M$ is fixed. Let us fix $t\in [0,T]$. Let $F^n$ be the mollifier for $F$ according to \eqref{mol}. Since functional $F$ satisfies the assumptions of Lemma \ref{preIto}, formula
\eqref{preIto1} holds for  $F^n(X^M)$. In the sequel, we will study the limit of each term in \eqref{preIto1} as $n\to\infty$. By A1.(ii), $F^n_t(X^M_t) \to F_t(X^M_t)~a.s$ as $n\rightarrow \infty$.

\

\noindent \textbf{STEP 1:} Let us prove the convergence
\aaa{
\label{h-conv}
 \int_0^t\nabla^h F^n_s(X^M_s)ds \to  \int_0^t\nabla^h F_s(X^M_s)ds \quad \text{a.s.}
 }
Lemma \ref{horderm} yields
\aaa{
\label{f15}
% F^n_t (c_t) = \int_0^{\frac2{n}} \rho(c(t)-y)  G^{y}_t(c_t) dy, \\
 \nab^h F^n_s (X^M_s) =  \int_{-\infty}^{\infty} \rho_n(X^M(s)-y)  \nab^h \mathcal{F}^y_s(X^M_s) dy; 0\le s\le T.
}
By Assumption A2, $y\mapsto \nabla^h\mathcal{F}^y_s(X^M_s)$ is continuous a.s.for each $s\in [0,T]$ and hence

\aa{
\lim_{n\to\infty} \nab^h F^n_s (X^M_s) =  \nab^h \mathcal{F}^{X^M(s)}(X^M_s) = \nab^h F_s(X^M_s)~a,s; 0\le s\le T.
}
From Assumption A2, $\{\nabla^h\mathcal{F}^y; y\in \mathbb{R}\}$ is state boundedness-preserving. Then, bounded convergence theorem yields

$$\lim_{n\rightarrow \infty}\int_0^t \nabla^h F^n(X^M_s)ds =\int_0^t \nabla^h F(X^M_s)ds~a.s$$

\

\noindent \textbf{STEP 2:} Next, we will prove that
\aaa{
\label{convergence}
\int_0^t \nab^v F_s^n(X^M_s) dX^M(s) \to \int_0^t \nab^-_x F_s(^{X^M(s)}X^M_s) dX^M(s)
\quad \text{in} \; L^2(\mathbb{P})
}
as $n\rightarrow \infty$. Firstly, we will show that under the assumptions of the theorem, for each fixed $(t,c,x)\in \mathcal G_M$, $\nab_x F_t^n(^xc_t)$ converges to $\nab_x^- F_t(^xc_t)$ as $n\to \infty$.
Fix a path $c\in C([0,t];\mathbb{R})$. At first, one should notice the left derivative $\nab^-_x F_t(^x c_t)$ is well defined for $x=\gm_t(c_t)$.
Indeed, we shall represent the functional $F_t(^xc_t)$ in the following form

\aaa{
\label{f1}
F_t(^xc_t) = \hat F_t(c_t;x) + \tilde F_t(c_t;x)
}
where
\aa{
 &\hat F_t(c_t;x) := F_t(^xc_t) + (\nab_x F_t(^{\gm_t(c_t)-}c_t) - \nab_x F_t(^{\gm_t(c_t)+}c_t))(x-\gm_t(c_t))^+,\\
 &\tilde F_t(c_t;x) := (\nab_x F_t(^{\gm_t(c_t)+}c_t) - \nab_x F_t(^{\gm_t(c_t)-}c_t))(x-\gm_t(c_t))^+.
}
It is easy to see that the function $x\mapsto \hat F_t(c_t;x)$ is $C^1$ in $x\in [-M,M]$. But on $[-M,\gm_c(t)]$,
$\hat F_t(c_t;x) = F_t(^xc_t)$. Hence, $\nab_x^- F_t(^xc_t)$ exists at the point $x=\gm_t(c_t)$, and therefore,
everywhere on $[-M,M]$. From the assumptions of the theorem, it is also clear that
$\nab_x^- F_t(^xc_t)$ is bounded on $\mathcal G_M$. Thus, we verified the assumptions of Lemma \ref{lem1}
with respect to the function $h\mto F_t(^{x-h}c_t)$ with $(t,c,x)$ being a parameter. This implies the boundedness
of the ratios $(F_t(^xc_t) - F_t(^{x-h}c_t))/h$. Hence, Lebesgue's bounded convergence theorem yields
$\nab_x F^n_t(^x c_t)$ is the mollifier for $\nab^-_x F_t(^x c_t)$:
\aaa{
\label{moll44}
\nab_x F^n_t(^x c_t) = \int_0^2 \rho(y) \nab^-_x F_t(^{x-\frac{y}{n}} c_t) dy.
}
From the assumptions of the theorem and the existence of $\nab_x^- F_t(^xc_t)$ at  $x=\gm_t(c_t)$, we know that $x \mto \nab_x^- F_t(^xc_t)$ is left continuous.
By the boundedness of $\nab_x^- F_t(^xc_t)$ on $\mathcal G_M$  and its left continuity in $x$, we obtain that
for each $(x,t,c)\in\mathcal G_M$,
$\nab_x F^n_t(^xc_t) \to \nab_x^- F_t(^xc_t)$ as $n\to \infty$ by Lebesgue's theorem.

Next, since $\nab_x^- F_t(^xc_t)$ is bounded on $\mathcal{G}_M$,
its mollifier $\nab_x F^n_t(^xc_t)$ is bounded on $\mathcal{G}_M$ by the same constant. In particular, there exists $C$ such that $|\nab^v F^n_s(X^M_s)|
= |\nab_x F^n_s(^{X^M(s)}X^M_s)|\le C$ for every $(\omega,s)\in \Omega\times [0,t]$.
Now the $L_2$-convergence \eqref{convergence} is implied by the semimartingale decomposition,
It\^o's isometry, and the bounded convergence theorem.

%To prove convergence \eqref{convergence}, let us decompose  $X = x(0)+ W +A$ where $W$ is a square-integrable martingale and $A$ is a bounded variation process with integrable variation $|A|$. Cauchy-Schwartz inequality yields
%\begin{small}
%\begin{eqnarray*}
%\Big[\E\big|\int_0^t [\nab_x F_s^n(^{X^M(s)}X^M_s) - \nab^-_x F_s(^{X^M(s)}X^M_s)] d|A|(s)\big| \Big]^2&\lt& \E |A|(t)\times \E\int_0^t  |\nab_x F_s^n(^{X^M(s)}X^M_s) \\
%& &\\
%&-& \nab^-_x F_s(^{X^M(s)}X^M_s)|^2 d|A|(s).
%\end{eqnarray*}
%\end{small}
%The right hand side converges to zero as $n\rightarrow \infty$ by bounded convergence theorem. By It\^o's isometry, we also have
%
%
%\begin{small}
%\begin{eqnarray*}
%\E\big|\int_0^t [\nab_x F_s^n(^{X^M(s)}X^M_s) - \nab^-_x F_s(^{X^M(s)}X^M_s)] dW(s)\big|^2
%&=& \E \int_0^t  |\nab_x F_s^n(^{X^M(s)}X^M_s)\\
%& &\\
%&-& \nab^-_x F_s(^{X^M(s)}X^M_s)|^2 d[W,W]_s.
%\end{eqnarray*}
%\end{small}
%Again, the right hand side converges to zero by bounded convergence theorem. This shows that (\ref{convergence}) holds.

In the sequel, to shorten notation we write $[X^M] = [X^M,X^M]$.

\

\noindent \textbf{STEP 3:} Lastly, we investigate the limit of $\frac{1}{2}\int_0^t\nabla^{v,2}F^n_s(X^M_s)d[X^M](s)$ as $n\rightarrow \infty.$
By applying mollification \eqref{mol}
in \eqref{f1}, we obtain $F^n_t(^xc_t) = \hat F^n_t(c_t;x) + \tilde F^n_t(c_t;x)$, where $\hat{F}^n_t(c_t;x):= (\rho_n \star \hat{F}_t(c_t; \cdot))(x)$ and $\tilde{F}^n_t(c_t;x):= (\rho_n \star \tilde{F}_t(c_t; \cdot))(x)$. Let us define $\hat F^n_t(c_t) := \hat F^n_t(c_t;c(t))$ and
$\tilde F^n_t(c_t) := \tilde F^n_t(c_t;c(t))$. We have:

\aaa{
\label{decomp}
\frac{1}{2}\int_0^t\nabla^{v,2}F^n_s(X^M_s)d[X^M](s) = \frac{1}{2}\int_0^t\nabla^{v,2}\hat F^n_s(X^M_s)d[X^M](s) +
\frac{1}{2}\int_0^t\nabla^{v,2}\tilde F^n_s(X^M_s)d[X^M](s).
}
Note that
%for each fixed $t\in [0,T]$ and $c\in C([0,T],\Rnu)$,
$\hat F_t(c_t;x)$ is $C^1$ in $x$  on $\mathcal G_M$
and the map $x\mto \nab_x \hat F_t(c_t;x)$ has on $\mathcal G_M^\gm$ a bounded left derivative
$\nab_x^-  \nab_x \hat F_t(c_t;x)$.
By Lemma \ref{lem1} and Remark \ref{rem1},
$x\mto \nab^2_x \hat F^n_t(c_t,x)$ is the mollifier for
$x\mto \nab_x^- \nab_x \hat F_t(c_t,x)$ on $[-M,\gamma_t(c_t) - \eps]$ for any sufficiently small $\eps>0$, i.e.

 \aaa{
 \label{f89}
 \nab^2_x \hat F^n_t(c_t;x) = \int_0^2 \rho(y) \nab_x^- \nab_x \hat F_t(c_t;x-\frac{y}{n})dy
 =  \int_0^2 \rho(y) \nab_x^{-,2} F_t(^{x-\frac{y}{n}}c_t)dy, \quad x\in [-M,\gamma_t(c_t) - \eps].
 }
 By assumption, $x\mto \nab_x^{-,2} F_t(^xc_t)$ is bounded on $\Gamma_{c,t}$ and
 its left limit exists at $x=\gm_t(c_t)$. This implies that \eqref{f89} holds for all $x\in [-M,\gm_t(c_t)]$.
 We note also that \eqref{f89} holds for $x\in (\gm_t(c_t) + \frac2{m},M]$ whenever $n>m$ and $m$
 is fixed arbitrary.
 By Lebesgue's theorem, we pass to the limit in \eqref{f89} as $n\to\infty$ while
 $x\in  [-M,\gm_t(c_t)] \cup (\gm_t(c_t) + \frac2{m},M]$ and $(t,c)\in [0,T]\times C([0,T],\Rnu)$ are fixed.
 We obtain that for $(x,c,t) \in [-M,\gm_t(c_t)] \cup (\gm_t(c_t) + \frac2{m},M] \times C([0,T],\Rnu) \times [0,T]$
 \aa{
 \lim_{n\to\infty} \nab^2_x \hat F^n_t(c_t;x) = \nab_x^{-,2} F_t(^{x-}c_t).
 }
 Since $m$ is fixed arbitrary, the above equality holds for all $(x,c,t) \in \mathcal G_M$.

Therefore, we have
\aa{
\lim_{n\to\infty}  \nabla^{v,2}\hat F^n_t(X^M_t) = \nab^{-,2}_x F_t(^{X^M(t)-}X^M_t)~a.s,
}
and
\aa{
\int_0^t\nabla^{v,2}\hat F^n_s(X^M_s)d[X^M](s) \to \int_0^t \nab^{-,2}_x  F_s(^{X^M(s)-}X^M_s) d[X^M](s)
\quad \text{a.s.}
}
by bounded convergence.

Let us investigate the convergence of the last term in \eqref{decomp}. It is convenient to introduce the following notation: We define $\gamma^M(s):=\gamma_{s\wedge T_M}(X^M_s)$ and $\tilde{X}^M(s) :=X^M(s) - \gamma^M(s); 0\le s\le T.$ Let $\ffi^n_s(x)$ be the  mollifier of $(x-\gm^M(s))^+$ according to formula \eqref{mol}, and let $\ffi^n(x)$ be the
mollifier of $x^+$. It is easy to verify that $\ffi^n_s(x) = \ffi^n(x-\gm^M(s))$.  Therefore,
\aa{
\td F^n_t(X^M_t;x) = &(\nab F_t(^{\gm^M(t)+}X^M_t) - \nab F_t(^{\gm^M(t)-}X^M_t)) \ffi^n_t(x) \\
= & (\nab F_t(^{\gm^M(t)+}X^M_t) - \nab F_t(^{\gm^M(t)-}X^M_t))
\ffi^n(x-\gm^M(t))~a.s.
}
Note that $\nab_x \ffi^n(x) = \int_{-\infty}^\infty \rho_n(x-y) H(y) dy$, where $H$ is the Heaviside function, and
that $\nab^2_x \ffi^n(x) = \int_{-\infty}^\infty \rho_n(x-y) dH(y) = \rho_n(x)$. Note that since
$\gm_\cdot(X_\cdot)$ has continuous bounded variation paths, then $[X^M](s) = [\td X^M](s)~a.s; 0\le s\le T$.

Now let $\td \ell^x_M$ be the local time of $\td X^M$. By Lemma \ref{occupationformula}, we obtain:
\aa{
&\frac12\int_0^t\nabla^{v,2}\tilde F^n_s(X^M_s)d[X^M](s) = \frac{1}{2} \int_0^t \nab^2_x \td F^n_s(X^M_s, \td X^M(s) + \gm^M(s))d[\td X^M](s)\\
 &= \frac{1}{2}\int_{-\infty}^\infty dx \int_0^t \nab^2_x \td F^n_s(X^M_s; x + \gm^M(s)) d_s\td  \ell^x_M(s) \\
&= \frac{1}{2}\int_{-\infty}^\infty \nab^2_x \ffi^n(x)  dx \int_0^t(\nab_x F_s(^{\gm^M(s)+}X^M_s) - \nab_x F_s(^{\gm^M(s)-}X^M_s))d_s\td  \ell^x_M(s)\\
&= \frac{1}{2}\int_{-\infty}^\infty \rho_n(x) dx   \int_0^t(\nab_x F_s(^{\gm^M(s)+}X^M_s) - \nab_x F_s(^{\gm^M(s)-}X^M_s))d_s\td \ell_M^x(s)\\
& \to \frac{1}{2}\int_0^t(\nab_x F_s(^{\gm^M(s)+}X^M_s) - \nab_x F_s(^{\gm^M(s)-}X^M_s))d_s\td \ell_M^{\,0}(s)~a.s.\quad \text{as} \; n\to \infty.
}
The above computations imply formula \eqref{formula5} at time $t\wedge T_M$. Letting $M$ go to infinity, we obtain \eqref{formula5}.
\end{proof}

The most simple application of Theorem~\ref{Ito4} is a pathwise description of the running maximum. A version of this formula appeared in Dupire \cite{dupire} but without a rigorous proof.

\begin{example}\label{runningmax}
\rm
Let us apply formula \eqref{formula5} to the running maximum $F_t(c_t)=\max_{s\in [0,t]} c(s); c\in\Lambda$.
One immediately verifies that $F$ satisfies the assumptions of Theorem \ref{Ito4}. Let us compute
each term of  \eqref{formula5}. We have:
$\nab^h F_t(X_t) = 0$, $\nab^-_x F_t(^xX_t) = 0$ if $x\lt F_t(X_t)$ and  $\nab^-_x F_t(^xX_t) = 1$ if $x > F_t(X_t)$.
In particular, $\nab^-_x F_t(^{X(t)}X_t) = 0$. Next, for $\gamma_t(c_t) = F_t(c_t)$, one can easily check that for each $c\in C([0,T];\mathbb{R})$, the function $x\mto F_t(^xc_t)$ is $C^1$ for $x\in (-\infty, \gamma_t(c_t)\cup (\gamma_t(c_t),+\infty)$ and $\nab^{-,2}_x F_t(^xc_t) = 0$ in this open set. Finally, we notice that $\nab_x F_t(^{\gamma_t(X_t)+}X_t) - \nab_x F_t(^{\gamma_t(X_t)-}X_t) = 1$  for all $t\in [0,T]$.
By formula \eqref{formula5},
\aa{
\sup_{0\le s\le t}X(s) = X(0) + \frac{1}{2}\td\ell^{\,0}(t),
}
where $\tilde{\ell}$ is the local time of the semimartingale $X(t) - \sup_{0\le s\le t}X(s); 0\le t\le T$.
\end{example}

Let us now apply Theorem \ref{Ito4} to concrete path-dependent functionals arising in Mathematical Finance.

\begin{example}\label{look1}
\rm
Similar to example~\ref{runningmax}, we shall also consider the payoff decomposition of a standard lookback option with fixed strike $K$ (see e.g Kwok~\cite{kwok} for further details). For a given constant $K>0$, we consider $F_t(c_t) = \big(\sup_{0\le s\le t}c(s)-K\big)^+$for $c\in \Lambda$. In this case, a straightforward application of Theorem~\ref{Ito4} yields

$$\Big(\underset{ 0\le s\le t}{\operatorname{sup}}X(s)-K\Big)^+ = \big( X(0)-K \big)^+ + \frac{1}{2}\td\ell^{\,0}(t); 0\le t\le T$$
where $\tilde{\ell}$ is the local time of the semimartingale $X(t) - \max\{\sup_{0\le s\le t}X(s); K\}; 0\le t\le T$.
\end{example}

\begin{example}
\rm
For each non-negative path $c\in \Lambda$, let us consider
$$F_t(c_t) = \big( c(t)-\lambda\inf_{T_0\le s\le t}c(s)\big)^+$$
where $\lambda>1$ and $0\le T_0 < T$ are arbitrary constants. This functional is the payoff of the so-called partial lookback european call option which allows lower investments than derivative contracts based on the payoff given in Example \ref{look1} (see e.g \cite{kwok}). Let us now apply Theorem \ref{Ito4} to give a novel representation for this payoff. For simplicity, we set $T_0=0$. Indeed, A1 (i), A1(ii) and A2 hold where $\nabla^h\mathcal{F}^x_t(c_t)=0$ for every $x\in \mathbb{R}_+$ and a non-negative path $c\in \Lambda$. By the very definition of $F$, it is apparent that the bounded variation functional which encodes the whole singularity is $\gamma_t(c_t)=\lambda\inf_{0\le s\le t}c(s); 0\le t\le T$. Moreover, $\nabla^-_xF_t(^{x}c_t) = 0$ if $x\le \gamma_t(c_t)$ and $\nabla^-_x F_t(^{x}c_t) = 1$ for $\gamma_t(c_t) < x$. In particular, $\nabla^- F_t(^{X(t)}X_t) = 1\!\!1_{\{ X(t) > \gamma_t(X_t)\}}; 0\le t \le T$. Moreover, $\nabla_x F_t(^{x}c_t) = 0$ if $x < \gamma_t(c_t)$ and $\nabla_x F_t(^{x}c_t) = 1$ if $x > \gamma_t(c_t)$. In particular, $\nabla^-_xF_t(^{\gamma_t(c_t)-}c_t) = 0$, $\nabla^-_xF_t(^{\gamma_t(c_t)+}c_t) = 1$ and $\nabla^{-,2}_xF_t(^{x}c_t) = 0$ over $(-\infty, \gamma_t(c_t))\cup (\gamma_t(c_t), +\infty)$. Finally, if $X$ is a non-negative square-integrable continuous semimartingale, then applying formula (\ref{formula5}), we get

$$\big( X(t) - \lambda \inf_{0\le s\le t}X(s)\big)^+ =  \int_0^t1\!\!1_{\{ X(s) > \gamma_s(X_s)\}}dX(s) + \frac{1}{2}\td\ell^{\,0}(t); 0\le t\le T,$$
where $\tilde{\ell}$ is the local time of the semimartingale $X(t)  -  \lambda\inf_{0\le s\le t}X(s); 0\le t\le T.$ and $\big( (1-\lambda)X(0)\big)^+=0$.
\end{example}

\section{$(p,q)$-bivariations and Functional It\^o formulas}\label{bivariationsection}
In this section, we provide an It\^o formula in the sense of Young in the path-dependent case. We refer the reader to the seminal work by Young~\cite{young1} for a full treatment of double Lebesgue-Stieljes-type integrals for unbounded variation functions. For a more simplified presentation, see e.g Ohashi and Simas~\cite{OS}.

Before presenting the main results, we recall some basic results from deterministic double integrals in Young-sense~\cite{young1}. Recall that if $f:[a,b]\rightarrow \mathbb{R}$ is a real-valued function and $p\ge 1$, then

$$\|f\|^p_{[a,b];p}:= \sup_{\Pi}\sum_{x_i\in \Pi}|f(x_i)-f(x_{i-1})|^p <\infty$$
where $\sup$ is taken over all partitions $\Pi$ of a compact set $[a,b]\subset \mathbb{R}$. The following notion is originally due to Young~\cite{young1} and it will play a key role in this section:
\begin{definition}\label{pq}
We say that $h:[a,b]\times [c,d]\rightarrow \mathbb{R}$ has $(p,q)$-bivariation for $p,q \ge 1$ if

$$\|h\|_{1;p}:=\sup_{y_1,y_2\in [c,d]^2}\| h(\cdot, y_1) - h(\cdot, y_2)\|_{[a,b];p}< \infty,$$
and
$$\|h\|_{2;q}:=\sup_{x_1,x_2\in [a,b]^2}\| h(x_1, \cdot) - h(x_2,\cdot)\|_{[c,d];q}< \infty.$$
\end{definition}

The importance of $(p,q)$-bivariation lies in the following result, which is a particular case of Theorem 6.3 due to Young~\cite{young1}.

\begin{theorem}[Theorem 6.3~\cite{young1}]\label{mainbound}
Let $h,G:[a,b]\times [c,d]\rightarrow \mathbb{R}$ be two functions, where $h$ vanishes on the lines $x=a$ and $y=c$ and has
bounded $(p,q)$-bivariation, and $G$ satisfies $|\Delta_i\Delta_jG(x_i,y_j)|\leq C|x_i-x_{i-1}|^{1/\tilde{p}}|y_j-y_{j-1}|^{1/\tilde{q}}$, for
some constant $C>0$, and $\tilde{p},\tilde{q}\ge 1$. If there exists $\alpha\in (0,1)$ such that
$$\alpha/p + 1/\tilde{p}>1\quad \hbox{and}\quad(1-\alpha)/q + 1/\tilde{q}>1,$$
then, the 2D Young integral $\int_a^b\int_c^d h(x,y)d_{(x,y)}G(x,y)$ exists.
\end{theorem}

\begin{remark}
We stress that there exists a related literature on $2D$-Young integral based on joint variations (see e.g~Friz and Victoir \cite{friz, friz1}) and related norms (see e.g~ Towghi \cite{towghi1}), rather than the bivariation concept. Indeed, one can check that $\|h\|_{1;p}\le RV^{p,p}_{[a,b]\times[c,d]}(h)$ and $\|h\|_{2;q}\le RV^{q,q}_{[a,b]\times [c,d]}(h)$ and these inequalities may be strict. See Section \ref{jointsection} for the definition of the norm $RV$.
\end{remark}

\begin{remark}
In general, we only know that generic continuous semimartingales admit local times with finite $(1,2+\delta)$-bivariation (for every $\delta >0$) rather than joint variation~(see Lemma 2.1 in Feng and Zhao~\cite{feng}). In some particular cases, the local time of a semimartingale admits joint variation. See Section \ref{jointsection} for details about symmetric stable processes.
\end{remark}

\subsection{Functional It\^o formula}\label{funcsec1}
Throughout this section, $\delta >0$ and $p, \tilde{p},\tilde{q}\ge 1$ are constants such that $\frac{1}{p} + \frac{1}{2+\delta }> 1$ and there exists $\alpha\in (0,1)$ such that
$$\alpha + \frac{1}{\tilde{p}}>1\quad \hbox{and}\quad\frac{(1-\alpha)}{2+\delta} + \frac{1}{\tilde{q}}>1$$

\begin{lemma}\label{intpartsformula}
Let $\varphi:\Omega\times [0,T]\times \mathbb{R}\rightarrow \mathbb{R}$ be a stochastic process such that $(t,x)\mapsto \nabla^2_x\varphi(\omega, t, x)\in  C([0,T]\times\mathbb{R};\mathbb{R})$ for each $\omega\in \Omega$ and $\nabla^2_x\varphi$ is bounded on $\Omega\times [0,T]\times[-M,M]$ for each $M>0$. Then,

%\item Assume for every $M>0$, there exists a constant $C$ such that

%\begin{equation}\label{p1}
%|\Delta_i\Delta_j\nabla_x\varphi (t_i,x_j)|\leq C|t_i-t_{i-1}|^{1/\tilde{p}}|x_j-x_{j-1}|^{1/\tilde{q}}~a.s
%\end{equation}
%for every partition $\{t_i\}_{i=0}^N\times \{x_j\}_{j=0}^{N^{'}}$ of $[0,T]\times [-M,M]$. Moreover,

%\begin{equation}\label{p2}
%\sup_{0\le t\le T}\|\nabla_x\varphi(t, \cdot)\|_{[-M,M]; p}< \infty~a.s
%\end{equation}
%\end{itemize}

\begin{eqnarray}
\nonumber \int_0^t\nabla^2_{x}\varphi(s, X^M(s))d[X^M,X^M](s) &=& \int_{\mathbb{R}}\Big(\int_0^{t\wedge T_M}\nabla^2_{x}\varphi(s,y)d_s\ell^y(s)\Big)dy\\
\nonumber& &\\
\label{p4}&=& -\int_0^{t\wedge T_M}\int_{\mathbb{R}}\nabla_x\varphi(s, x)d_{(s,x)}\ell^x(s)~a.s
\end{eqnarray}
for $0\le t\le T$. In (\ref{p4}), the double integral is interpreted as a 2D Young integral in the sense of \cite{young1}.
\end{lemma}
\begin{proof}
Let us fix $M>0$, $t\in [0,T]$ and $\omega\in \Omega$. In the sequel, we omit the variable $\omega$ in the computations. At first, we recall that if $\nabla^2_{x}\varphi:\Omega\times [0,T]\times [-M,M]\rightarrow \mathbb{R}$ is bounded, then Lemma \ref{occupationformula} yields

\begin{equation}\label{occ1}
\int_0^{t}\nabla^2_{x}\varphi(s,X^M(s))d[X^M,X^M](s) = \int_{\mathbb{R}}\Big(\int_0^{t\wedge T_M}\nabla^2_{x}\varphi(s,y)d_s\ell^y(s)\Big)dy.
\end{equation}
Let $0=t_1 < t_2 < \ldots \le t_{m+1}=t\wedge T_M$ and $-L=x_1 < x_2< \ldots < x_{n+1}=L$ where $[-L,L]$ is a compact set. Let us fix $\omega\in \Omega$. Since the local-time has compact support, we stress that we can always add some points in the partition in such way that $\ell^{x_1}(t_j,\omega) = 0$ and $\ell^{x_{n+1}}(t_j, \omega)=0$ for every $j=1, \ldots, m$. To keep notation simple, we write $\varphi=\varphi(\omega)$ and $\ell=\ell(\omega)$. Mean value theorem allows us to argue just like in Remark 1 in \cite{feng} to get the following identity

\begin{equation}\label{riemsum1}
\sum_{i=1}^m\sum_{j=1}^n\nabla_x\varphi(t_j,x_i)\big(\Delta_j\ell^{x_{i+1}}(t_{j+1}) - \Delta_i\ell^{x_{i}}(t_{j+1})\big) = - \sum_{i=1}^m\sum_{j=1}^n\nabla^2_x\varphi(t_j,y_i)\Delta_j\ell^{x_{i+1}}(t_{j+1})(x_{i+1} - x_i)
\end{equation}
where $x_i < y_i < x_{i+1}; i=1, \ldots, m$. Let $K$ be the compact support of $x\mapsto \ell^x(T)$. We notice that the function $x\mapsto \sum_{j}\varphi(t_j, x)\Delta_j\ell^x(t_{j+1})$ is cadlag and hence almost everywhere continuous. The boundedness assumption yields

\begin{eqnarray*}
\lim_{n\rightarrow \infty}\lim_{m\rightarrow \infty}\sum_{i=1}^m\sum_{j=1}^n\nabla^2_x\varphi(t_j,y_i)\Delta_j\ell^{x_{i+1}}(t_{j+1})(x_{i+1} - x_i)&=&\lim_{n\rightarrow \infty} \int_K\sum_{j=1}^n\nabla^2_x\varphi(t_j, x)\Delta_j \ell^x(t_{j+1})dx\\
& &\\
&=& \int_{\mathbb{R}}\int_0^t\nabla^2_x\varphi(s,x)d\ell^x_sdsdx.
\end{eqnarray*}
From (\ref{riemsum1}), we conclude the proof.
%In one hand, by performing a similar computation given in identity (4.5) in Feng and Zhao \cite{feng}, we do have

%\begin{eqnarray}
%\label{expand}\sum_{i=0}^{m-1}\sum_{j=0}^{m^{'}-1}\nabla_{x}\varphi(s_j, x_i)\Delta_i\Delta_j %\ell^{x_{i+1}}(s_j+1)&=&\sum_{i=1}^{m}\sum_{j=1}^{m^{'}}\ell^{x_{i}}(s_j)  \Delta_i\Delta_j\nabla_{x}\varphi(s_j, x_i)\\
%\nonumber& &\\
%\nonumber&-& \sum_{i=1}^{m}\ell^{x_i}(t)\Delta_i \nabla_x\varphi(t,x_i)
%\end{eqnarray}
%for any partition $\{s_j\}_{j=0}^{m^{'}}\times \{x_i\}_{i=0}^m$ of $[0,t\wedge T_M]\times [-M,M]$. From Lemma 2.1 in \cite{feng}, we know that the local-time of $X$ has $1,2+\delta$-bivariation. Moreover, we shall assume that $\ell^{-M}_\cdot=0~a.s$ and obviously we have $\ell^\cdot(0)=0~a.s$. Then (\ref{p1}) and (\ref{p2}) allow us to apply Corollary \ref{mainbound} and the classical theorem on the existence of one dimensional Young integrals to conclude that

%\begin{eqnarray}
%\nonumber\int_0^{t\wedge T_M}\int_{[-M,M]}\nabla_x\varphi(s,x)d_{s,x}\ell^x(s) &=& \int_0^{t\wedge T_M}\int_{[-M,M]}\ell^x(s)d_{s,x}\nabla_x\varphi(s,x)\\
%\label{p3} & &\\
%\nonumber&-& \int_{[-M,M]}\ell^x(t\wedge T_M) d_x\nabla_x\varphi(t\wedge T_M,x)
%\end{eqnarray}
%Next, we shall expand the double Riemman sum (\ref{expand}) just like in Remark 1 in \cite{feng}. By using the man value theorem, we may conclude that

\end{proof}

Let us now assume additional hypotheses on the functional $F$ to shift quadratic variation to local-time integrals.

\

\noindent \textbf{Assumption B:} The spatial weak derivative $(\nabla^w_x F_t)(^{x}c_t)$ satisfies: For every $L>0$, there exists a constant $C$ such that

\begin{equation}\label{p5}
|\Delta_i\Delta_j(\nabla^w_x F_{t_i})(^{x_j}c_{t_i})|\leq C|t_i-t_{i-1}|^{1/\tilde{p}}|x_j-x_{j-1}|^{1/\tilde{q}}
\end{equation}
for every partition $\{t_i\}_{i=0}^N\times \{x_j\}_{j=0}^{N^{'}}$ of $[0,T]\times [-L,L]$ and $c\in C([0,T];\mathbb{R})$. Moreover,
\begin{equation}\label{p6}
\sup_{0\le t\le T}\|(\nabla^w_x F_t)(^{\cdot}c_t)\|_{[-L,L]; p}< \infty
\end{equation}
for every $c\in C([0,T];\mathbb{R})$.
\

In the sequel, we provide a mild hypothesis to get convergence of local-time and stochastic integrals.

\

\noindent \textbf{Assumption C:} We assume piecewise uniform left-continuity in the following sense: For every $\varepsilon> 0,M >0$ and $c\in C([0,T]; [-M,M])$ there exists $\{x_{i}\}_{i=0}^{n+1}$, $-M=x_0 < x_1 < \ldots < x_n < x_{n+1} =M$ such that

$$\sup_{0\le t\le T}|(\nabla^w_y F_t)(^{y}c_t) - (\nabla^w_xF_t)(^{x}c_t)|< \varepsilon $$
whenever $x_0 \le y\le x\le x_1$ or $x_i < y \le x \le x_{i+1}; i=1, \ldots,n$.

%$$\lim_{\delta\rightarrow 0+}\sup_{\substack{0\le t\le T \\ x\in [-M,M]}}|\nabla^w_x F_t(^{x-\delta}c_t) - \nabla^w_x F_t(^{x}c_t)|=0  $$
%for each $c\in C([0,T]; K)$, a compact set $K\subset \mathbb{R}$ and $0 < M < \infty$.

\
An immediate consequence of Lemma \ref{intpartsformula} is the following remark.
\begin{corollary}\label{Itoformulamoll}
If $F$ satisfies Assumptions A1(i) and A2, then for each $M>0$ and $n\ge 1$,
$$F^n_t(X^M_t) = F^n_0(X^M_0) + \int_0^{t}\nabla^h F^n_s(X^M_s)ds + \int_0^{t\wedge T_M}\nabla^v F^n_s(X_s)dX(s) -\frac{1}{2}\int_0^{t\wedge T_M}\int_{\mathbb{R}}\nabla_x F^n_s(^{x}X_s)d_{(s,x)}\ell^x(s)$$
~a.s.for $0\le t\le T.$
\end{corollary}
\begin{proof}
Let us fix $M>0$ and $n\ge 1$. In one hand, $\rho^{(2)}$ has compact support and $^{x}X^M_\cdot \in D([0,T]; [-M,M])$ a.s, then we shall use Assumption A1(i), to state that $(\omega,t,x)\mapsto \nabla^2_{x}F^n_t(^{x}X^M_t(\omega))$ is a bounded measurable process on $\Omega\times [0,T]\times [M,M]$. On the other hand, $\nabla^{v,2}F^n_t(X^M_t) = \nabla^2_{x}F^n_t(^{x}X^M_t)|_{x=X^M(t)}$ so that (\ref{p4}) in Lemma \ref{intpartsformula} yields

\begin{eqnarray*}
\int_0^t\nabla^{v,2}F^n_s(X^M_s)d[X^M,X^M](s) &=& -\int_0^{t\wedge T_M}\int_{\mathbb{R}}\nabla_x F^n_s(^{x}X^M_s)d_{(s,x)}\ell^x(s)\\
 & &\\
 &=& -\int_0^{t\wedge T_M}\int_{\mathbb{R}}\nabla_x F^n_s(^{x}X_s)d_{(s,x)}\ell^x(s)~a.s; 0\le t\le T.
 \end{eqnarray*}
Lemma \ref{preIto} allows us to conclude the proof.

%By the very definition,

%$$\nabla^{v,2} F^n_t(X^M_t) = n^2\int_{[-M,M]}\rho^{(2)}\big(n(X^M(t)-y) \big) F_t(^{y}X^M_t)dy$$
\end{proof}

Now we are able to present the main result of this section. It extends~\cite{feng} in the context of path-dependent functionals as well as Th. 8.1 in~\cite{LOS} in the context of generic semimartingales. In particular, it complements the results given in section~\ref{smoothcase} when $x\mapsto (\nabla^w_xF_t)(^{x}c_t)$ has bounded variation.

\begin{theorem}\label{thYoung}
Let $F$ be a functional satisfying Assumptions A1, A2, B and C. Then

\begin{equation}\label{formula1}
F_t(X_t) = F_0(X_0)+ \int_0^t\nabla^hF_s(X_s)ds + \int_0^t\nabla^wF_s(X_s)dX(s) - \frac{1}{2}\int_{-\infty}^{+\infty}\int_0^t(\nabla^w_xF_s)(^{x}X_s)d_{(s,x)}\ell^x(s)~a.s
\end{equation}
for $0\le t\le T$.
\end{theorem}

\begin{proof}
Let $M>0$ be such that $supp~\rho\subset [-M,M]$. To keep notation simple, we set $t_M:=t\wedge T_M$. At first, we claim that the following convergence holds

\begin{equation}\label{step1}
\int_0^{t_M}\int_{[-M,M]}\nabla_x F^n(^{x}X_s)d_{(s,x)}\ell^x(s)\rightarrow \int_0^{t_M}\int_{[-M,M]}(\nabla^w_x F_s)(^{x}X_s)d_{(s,x)}\ell^x(s)
\end{equation}
almost surely as $n\rightarrow \infty$, for each $t\in [0,T]$. Indeed, by making a change of variable

\begin{equation}\label{step2}
\nabla_x F^n_t(^{x}X_t) - (\nabla^w_xF_t)(^{x}X_t) = \int_{-M}^M\rho(z)\Big((\nabla^w_xF_t) (^{x-z/n}X_t) - (\nabla^w_xF_t)(^{x}X_t)\Big)dz~a.s
\end{equation}
for every $(t,x)\in [0,T]\times [-M,M]$.
Let us fix $\omega\in \Omega$. By Assumption C, we then have

%By assumption, $\nabla^wF:\Lambda\rightarrow \mathbb{R} $ is $\Lambda$ - uniformly continuous so that for a given $\varepsilon >0$, there exists $\delta>0$ such that whenever

%\begin{eqnarray*}
%d\Big( (t,^{x-z/n}X_t(\omega)), (t, ^{x}X_t(\omega))   \Big)&=&|t-t| + \sup_{0\le u\le t}|^{x-z/n}X_t(\omega,u) - ^{x}X_t(\omega,u)|\\
%& &\\
%&=& |\frac{z}{n}| < \delta
%\end{eqnarray*}
%we have  $\big|\nabla^w_xF_t (^{x-z/n}X_t(\omega)) - \nabla^w_xF(^{x}X_t(\omega))\big|< \varepsilon$. In this case, we shall take $n$ large enough so that $|z/n|<\delta$ uniformly in $(x,z)\in [-M,M]\times [-M,M]$ and $\omega\in \Omega$.  Hence, for a given $\omega\in \Omega$, identity (\ref{step2}) yields

\begin{equation}\label{step3}
\sup_{(x,t)\in [-M,M]\times [0,T]}\big|\nabla_x F^n_t(^{x}X_t(\omega)) - (\nabla^w_xF_t)(^{x}X_t(\omega))\big|\rightarrow 0
\end{equation}
as $n\rightarrow \infty$. Moreover, for any partition $\{t_{i}\}_{i=0}^N\times \{x_j\}_{j=0}^{N^{'}}$ of $[0,t]\times [-M,M]$, we have

\begin{equation}\label{step5}
|\Delta_j\Delta_i \nabla_x F^n_{t_i}(^{x_j}X_{t_i})|\le \int_{-M}^M \rho(z) |\Delta_j\Delta_i(\nabla^w_x F_{t_i})(^{x_j -\frac{z}{n}}X_{t_i})|dz~a.s.
\end{equation}
Let us now fix an arbitrary partition $\{t_{i}\}_{i=0}^N\times \{x_j\}_{j=0}^{N^{'}}$ of $[0,t]\times [-M,M]$. Let $\mathcal{P}_{[-Q,Q]}$ be the set of all partitions of the interval $[-Q,Q]$ for $0< Q< \infty$. We notice that for each $z\in [0,M]$ the set $\{x_j - z/n; j=0, \ldots, N'\}$ is a partition of $[-M - z/n, M-z/n]$. In particular,  $[-M-\frac{z}{n}, M - \frac{z}{n}]\subset [-M-\frac{M}{n}, M + \frac{M}{n}]$ for every $z\in [0,M]$ and $n\ge 1$. Then we shall find a compact set $[-2M,2M]$ such that $[-M-\frac{M}{n}, M + \frac{M}{n}]\subset [-2M,2M]~\forall n\ge 1$. More importantly. we shall add finitely many points in the set $\{x_j - z/n; j=0, \ldots, M, z\in [0,M]\}$ in such way that this can be viewed as a subset of $\mathcal{P}_{[-2M,2M]}$. A similar argument holds for $z\in [-M,0]$. Therefore, Assumption B and (\ref{step5}) yield the existence of a positive constant $C$ which only depends on $M>0$ such that

\begin{eqnarray}
\nonumber |\Delta_j\Delta_i \nabla_x F^n_{t_i}(^{x_j}X_{t_i})|&\le& \int_{-M}^M \rho(z) |\Delta_j\Delta_i(\nabla^w_x F_{t_i})(^{x_j -\frac{z}{n}}X_{t_i})|dz\\
\nonumber& &\\
\label{step6}&\le& C|t_{i}-t_{i-1}|^{1/\tilde{p}}|x_j-x_{j-1}|^{1/\tilde{q}}~a.s
\end{eqnarray}
for every $n\ge 1$. Let us fix $\omega\in \Omega^*$ and $t\in [0,T]$, where $\mathbb{P}(\Omega^*)=1$. We may suppose that $\ell^{-M}(\cdot, \omega) = 0$ and we obviously have $\ell^\cdot(0, \omega)=0$. Then, we shall apply Th 6.4 in Young \cite{young} to state that

\begin{equation}\label{step4}
\int_0^{t_M(\omega)}\int_{[-M,M]}\ell^x(\omega,s)d_{(s,x)}\nabla_x F^n_s(^{x}X_s(\omega))\rightarrow \int_0^{t_M(\omega)}\int_{[-M,M]}\ell^x(\omega,s)d_{(s,x)}(\nabla^w_xF_s)(^{x}X_s(\omega))
\end{equation}
as $n\rightarrow \infty$. Moreover,

\begin{eqnarray*}
\Delta_j \nabla_x F^n_{t_M(\omega)}(^{x_j}X_{t_M(\omega)}(\omega)) &=& \int_{-M}^M\rho(z)\Big[(\nabla^w_xF_{t_M(\omega)})(^{x_j-z/n}X_{t_M(\omega)}(\omega))\\
& &\\
&-& (\nabla^w_xF_{t_M(\omega)})(^{x_{j-1} - z/n}X_{t_{M}(\omega)}(\omega))\Big]dz
\end{eqnarray*}
Since $\int_0^2\rho(z)dz=1$, we shall apply Jensen inequality to get

\begin{equation}\label{step7}
|\Delta_j \nabla_x F^n_{t_M(\omega)}(^{x_j}X_{t_M(\omega)}(\omega))|^p \le \int_{-M}^M |\Delta_j(\nabla^w_x F_{t_M(\omega)})(^{x_j-z/n}X_{t_M(\omega)}(\omega))|^p\rho(z)dz
\end{equation}
The same argument used in (\ref{step6}) also applies here. In this case, by applying (\ref{p6}) into (\ref{step7}), we can find a compact set $[-Q,Q]$ such that

\begin{eqnarray}
\nonumber\sum_{j=0}^{N^{'}}|\Delta_j \nabla_x F^n_{t_M(\omega)}(^{x_j}X_{t_M(\omega)}(\omega))|^p &\le& \int_{-M}^M \sum_{j=0}^{N^{'}}|\Delta_j(\nabla^w_x F_{t_M(\omega)})(^{x_j - z/n}X_{t_M(\omega)}(\omega))|^p\rho(z)dz\\
\nonumber& &\\
\label{step8}&\le& \int_{-M}^M\| (\nabla^w_xF_{t_M(\omega)})(^{\cdot}X_{t_M(\omega)}(\omega))\|^p_{[-Q,Q]; p}\rho(z)dz\\
\nonumber& &\\
\nonumber &=&\| (\nabla^w_x F_{t_M(\omega)})(^{\cdot}X_{t_M(\omega)}(\omega))\|^p_{[-Q,Q]; p}
\end{eqnarray}
for every $n\ge 1$. Estimate (\ref{step8}) yields

\begin{equation}\label{step9}
\|\nabla_x F^n_{t_M(\omega)}(^{\cdot}X_{t_M(\omega)}(\omega))  \|^p_{[-M,M]; p}\le \| (\nabla^w_x F_{t_M(\omega)})(^{\cdot}X_{t_M(\omega)}(\omega))\|^p_{[-Q,Q]; p}
\end{equation}
for every $n\ge 1$. Estimate (\ref{step9}) together with (\ref{step3}) allow us to use Proposition 6.12 in e.g \cite{friz} to get

\begin{equation}\label{step10}
\int_{[-M,M]}\ell^x(t_M(\omega))d_x \nabla_x F^n_{t_M(\omega)}(^{x}X_{t_M(\omega)}(\omega))\rightarrow \int_{[-M,M]} \ell^x(t_M(\omega))d_x(\nabla^w_x F_{t_M(\omega)})(^{x}X_{t_M(\omega)}(\omega))
\end{equation}
as $n\rightarrow \infty$. By writing

\begin{eqnarray*}
\int_0^{t_M(\omega)}\int_{-M}^M\nabla_x F^n_s(^{x}X_s(\omega))d_{(s,x)}\ell^x(s) &=&\int_0^{t_M(\omega)}\int_{[-M,M]}\ell^x(\omega,s)d_{(s,x)}\nabla_x F^n_s(^{x}X_s(\omega)) \\
& &\\
- \int_{[-M,M]}\ell^x(t_M(\omega))d_x \nabla_x F^n_{t_M(\omega)}(^{x}X_{t_M(\omega)}(\omega))
\end{eqnarray*}
and using (\ref{step10}) and (\ref{step4}), we conclude that (\ref{step1}) holds. From Assumptions A1(ii), we know that $\lim_{n\rightarrow \infty}F^n_t(X^M_t) = F_t(X^M_t)~a.s; 0\le t\le T$. From Corollary \ref{Itoformulamoll}, it only remains to check that

\begin{equation}\label{step11}
\int_0^t\nabla^h F^n_s(X^M_s)ds\rightarrow \int_0^t\nabla^hF_s(X^M_s)ds
\end{equation}

\begin{equation}\label{step12}
\int_0^{t_M}\nabla^v F^n_s(X_s)dX(s)\rightarrow \int_0^{t_M}\nabla^w F_s(X_s)dX(s)
\end{equation}
in probability as $n\rightarrow\infty$. We have already checked that convergence (\ref{step11}) holds in the proof of Theorem \ref{Ito4}.

From (\ref{step3}), we know that for each $\omega\in \Omega$

$$\sup_{0\le t\le T}|\nabla^v F^n_t (X_t(\omega)) - \nabla^w F_t (X_t(\omega))|\rightarrow 0$$
as $n\rightarrow \infty$ so that

$$\int_0^{t_M}|\nabla^v F^n_s (X_s) - \nabla^w F_s (X_s)|^2d[X,X](s)\rightarrow 0  $$
in probability as $n\rightarrow \infty$. This shows that (\ref{step12}) holds. Summing up the above result together with Corollary \ref{Itoformulamoll}, we get

$$F_{t_M}(X_{t_M}) = F_0(X_0)+ \int_0^{t}\nabla^hF_s(X^M_s)ds + \int_0^{t_M}\nabla^wF_s(X_s)dX(s) - \frac{1}{2}\int_{-M}^{M}\int_0^{t_M}(\nabla^w_xF_s)(^{x}X_s)d_{(s,x)}\ell^x(s)$$
a.s.for $0\le t\le T$. By letting $M\rightarrow \infty$ and using the fact that $(x,t)\mapsto \ell^x(t)$ has compact support a.s, then we recover (\ref{formula1}).
\end{proof}

\begin{example}\label{expexample}
\rm
We consider an example studied by Le\~ao, Ohashi and Simas \cite{LOS} given by
$$F_t(c_t)= \int_{-\infty}^{c(t)}\int_0^t\varphi(c(s),y)dsdy; 0\le t\le T,$$
for $c\in \Lambda$, where $\varphi:\mathbb{R}^2\rightarrow \mathbb{R}$ is a two-parameter H\"{o}lder continuous function satisfying the following hypotheses:

\

\noindent \textbf{(i)} For every compact set $K\subset \mathbb{R}$, there exist constants $M_1$ and $M_2$ such that for every $a,z\in K$,
$$|\varphi(a,x)-\varphi(a,y)|\le M_1|x-y|^{\gamma_1},$$
and
$$|\varphi(c,z)-\varphi(d,z)|\le M_2|c-d|^{\gamma_2},$$
where $\gamma_1\in \left(\frac{1+\delta}{2+\delta},1\right]$, $\gamma_2\in (0,1]$ and $\delta>0$.

\

\noindent \textbf{(ii)} For every compact set $V_1\subset \mathbb{R}$ there exists a compact set $V_2$
such that $\{x; \varphi(a,x)\neq 0\}\subset V_2$ for every $a\in V_1$.

\

\noindent \textbf{(iii)} For every continuous path $c\in C([0,T];\mathbb{R})$, $\int_{[0,T]\times \mathbb{R}}|\varphi(c(s),y)|dsdy< \infty$.

This example was studied in \cite{LOS} in the Brownian filtration context where the authors show that it is a weakly differentiable process. One can easily check if $(\textbf{i, ii, iii})$ are in force, then this functional satisfies the assumptions in Theorem~\ref{thYoung}. In particular, if $X$ is a continuous semimartingale, the following decomposition holds

\begin{eqnarray*}
F_t(X_t) &=& F_0(X_0) + \int_0^t\int_0^s\varphi(X(r),X(s))drdX(s) + \int_0^t\int_{-\infty}^{X(s)}\varphi(X(s),y)dyds\\
& &\\
&-&\frac{1}{2}\int_0^t\int_{-\infty}^{+\infty}\int_0^s\varphi(X(r),x)drd_{(s,x)}\ell^x(s);~a.s, 0\le t\le T.
\end{eqnarray*}

One can also think in more general functionals of the form

$$\int_{-\infty}^{c(t)}Z_t(c_t ;y)dy; c\in \Lambda$$
where $Z = \{Z_t(\cdot; x):C([0,t];\mathbb{R})\rightarrow \mathbb{R}; x\in \mathbb{R}; 0\le t \le T\}$ is a family of functionals with suitable two-parameter H\"{o}lder regularity. See Example \ref{FPS}.
\end{example}

\section{Functional It\^o formula for symmetric stable processes under joint variation conditions}\label{jointsection}
In this section, we investigate It\^o formulas under different (and somewhat weaker) assumptions from the particular 2D-control given by (\ref{p5}) in Assumption B. In the language of rough path theory, assumption~(\ref{p5}) precisely says that if $\tilde{q}=\tilde{p}=\beta$ then $\nabla_x F_t(^{x}X_t)$ admits a 2D-control $\omega([t_1,t_2]\times [x_1,x_2]) = |t_1-t_2|^{\frac{1}{\beta}}|x_1-x_2|^{\frac{1}{\beta}}$ so that~(\ref{p5}) trivially implies that $(t,x)\mapsto \nabla_x F_t(^{x}X_t)$ has $(\beta,\beta)$-joint variation in the sense of Friz and Victoir~\cite{friz}. If the semimartingale local time $\{\ell^x(t); -L\le x\le L, 0\le t\le T\}$ admits joint variation over compact sets $[-L,L]\times [0,T]$ a.s. (see Definition \ref{jointdef}), then (\ref{p5}) and (\ref{p6}) in Assumption B can be weakened to more general types of controls.

To our best knowledge, it is only known that local-times associated to general continuous semimartingales admit finite $(1,2+\delta)$-bivariation a.s.for any $\delta>0$. This result is due to Feng and Zhao \cite{feng}. In the sequel, we study joint variation of local-times of semimartingales in the following sense.

\begin{definition}\label{jointdef}
Let $p,q,r,s\in [ 1, \infty)$, $-\infty <  a_1 < a_2 < +\infty$ and $-\infty < b_1 < b_2 <  \infty$. A function $H:[a_1,a_2]\times [b_1,b_2]\rightarrow \mathbb{R}$ has joint right finite $(p,q)$-variation when
$$RV^{p,q}_{[a_1,a_2]\times [b_1,b_2]}(H): = \sup_{\pi}\Bigg\{\Bigg[\sum_{i=1}^n\Big[\sum_{j=1}^m|\Delta_i\Delta_j H(t_i,x_j)|^p \Big]^{\frac{q}{p}} \Bigg]^{\frac{1}{q}} \Bigg\}< \infty.$$
It has joint left finite $(r,s)$-variation when
$$LV^{r,s}_{[a_1,a_2]\times [b_1,b_2]}(H): = \sup_{\pi}\Bigg\{\Bigg[\sum_{j=1}^m\Big[\sum_{i=1}^n|\Delta_i\Delta_j H(t_i,x_j)|^r \Big]^{\frac{s}{r}} \Bigg]^{\frac{1}{s}} \Bigg\}< \infty$$
where sup varies over all partitions $\pi:=\{a_1=t_1 \le t_2 \le \ldots \le t_n=a_2\}\times \{b_1=x_0\le x_1 \ldots \le x_m=b_2\}$ of $[a_1,a_2]\times [b_1, b_2]$.
\end{definition}
See Towgui~\cite{towghi} for more details on this variation concept. When $p=q$, this type of variation has been studied in the context of Gaussian rough paths (see e.g \cite{hairer} and \cite{friz2}). The following result is an immediate consequence of a fundamental estimate due to Marcus and Rosen~\cite{markus} in Lemma 3.3.

\begin{lemma}
\label{lem22}
Let $X$ be a real-valued symmetric stable process with exponent $1 < \beta\le 2$. Then for every natural number $p\ge 1$, there exists a positive number $C$ wich only depends on $(\beta,p)$ such that

\begin{equation}
\|\ell^x(t) - \ell^y(t) - \big(\ell^x(s) - \ell^y(s)\big)\|_{L^{2p}(\mathbb{P})}\le C |x-y|^{\frac{\beta-1}{2}}|t-s|^{\frac{\beta-1}{2\beta}}
\end{equation}
for any list of numbers $(t, s, x, y)\in \mathbb{R}^2_+\times \mathbb{R}^2.$
\end{lemma}
\begin{proof}
From Lemma 3.3 in \cite{markus}, we know there exists a constant $C>0$ which only depends on $(p,\beta)$ such that

\begin{equation}\label{markusbound}
\Big(\mathbb{E}|\ell^x(t)-\ell^y(t)|^{2p}\Big)^{\frac{1}{2p}}\le C|x-y|^{(\frac{\beta-1}{2})}t^{\frac{\beta-1}{2\beta}}
\end{equation}
for every $(t,x, y)\in \mathbb{R}_+\times \mathbb{R}^2.$ Let $\theta_t:\Omega\rightarrow\Omega$ be the standard shift operator defined by the relation $Y\circ \theta_t: = Y(\theta_t); t\ge 0$ for any random variable $Y$. Since $X$ is a Markov process, then we know that the associated local-time process $\{\ell^x(t); (x,t)\in \mathbb{R}\times\mathbb{R}_+\}$ is an additive functional. Hence, by using the Markov property and (\ref{markusbound}), if $(s,t,x,y)\in \mathbb{R}^2_+ \times\mathbb{R}^2$, then

\begin{eqnarray*}
\|\ell^x(t) - \ell^y(t) - \big(\ell^x(s) - \ell^y(s)\big)\|^{2p}_{L^{2p}(\mathbb{P})} &=& \mathbb{E}|\ell^x(t) - \ell^y(t) - \ell^x(s) + \ell^y(s)|^{2p}\circ \theta_s\\
& &\\
&=&\int_{-\infty}^{+\infty}\mathbb{E}|\ell^{x-v}(t-s)-\ell^{y-v}(t-s)|^{2p}\mathbb{P}_{X(s)}(dv)\\
& &\\
&\le& C|x-y|^{(\frac{\beta-1}{2})2p}|t-s|^{(\frac{\beta-1}{2\beta})2p}
\end{eqnarray*}
where $\mathbb{P}_{X(s)}$ is the law of $X(s)$.

\end{proof}

We are now able to show the following result.

\begin{lemma}\label{ltimejointvar}
Let $X$ be a stable symmetric process with exponent $1<  \beta \le 2$. Then for every compact subset $[-L,L]\subset \mathbb{R}$, the associated local time process $\ell$ of $X$ satisfies $RV^{\al_1,\al_2}_{[0,T]\times [-L,L]}(\ell) + LV^{\al_2,\al_1}_{[0,T]\times [-L,L]}(\ell) < \infty$ a.s. for any
$\al_1> \frac2{\beta-1}$ and $\al_2> \frac{2\beta}{\beta-1}$.
\end{lemma}
\begin{proof}
Let us fix a compact set $[-L,L]\subset \mathbb{R}$ and let $p\ge 1$ be an arbitrary positive integer. Theorem 3.1 from \cite{Hu} and Lemma \ref{lem22} imply that for every
 $\gm_1$ and $\gm_2$ satisfying
\aaa{
\label{ppp}
\gm_1 < \frac{\beta-1}2   -\frac1{2p} \quad  \text{and}  \quad \gm_2 < \frac{\beta-1}{2\beta} - \frac1{2p}
}
there exists a non-negative random variable $C_p(\om)$, which depends on $p$, such that
 \aaa{
\label{ineq}
|\ell^x(\omega,t) - \ell^y(\omega,t) - (\ell^x(\omega,s)- \ell^y(\omega,s))| \lt C_p(\om) |x-y|^{\gm_1}|t-s|^{\gm_2}
}
for every $s,t\in [0,T]$ and almost all $\omega\in \Omega$. In other words, for each pair of positive constants $\gm_1$ and $\gm_2$ satisfying

%Since $p$ is an arbitrary natural number, we can choose it sufficiently large
%so that for each pair of positive constants $\gm_1$ and $\gm_2$ satisfying

\aa{
\gm_1 < \frac{\beta-1}2   \quad  \text{and}  \quad \gm_2 < \frac{\beta-1}{2\beta}
}
there exists $p\ge 1$ which realizes (\ref{ppp}) and a non-negative random variable $C_p(\om)$, depending on $p$, such that \eqref{ineq} holds.

Now let $(\al_1,\alpha_2)$ be any pair of numbers satisfying $\al_1>\frac2{\beta-1}$ and $\al_2>\frac{2\beta}{\beta-1}$. Inequality \eqref{ineq} is fulfilled
for $\gm_1 = \al_1^{-1}$ and $\gm_2 = \al_2^{-1}$ and for a non-negative random variable $C_p(\om)$.  For a given partition,
$\pi =  \{-L=x_0\le x_1 \ldots \le x_m=L\} \times \{0=t_1 \le t_2 \le \ldots \le t_n=T\}$ of $[-L, L] \times [0,T]$, we then have

\aa{
\sum_{j=1}^m |\Dl_j\Dl_i \ell^{x_j}(\omega,t_i)|^{\al_1} \lt 2LC_p(\om)^{\al_1}|t_i - t_{i-1}|^{\gm_2\al_1},
}
and hence,
\aa{
\Big[\sum_{i=1}^n\big[\sum_{j=1}^m |\Dl_j\Dl_i \ell^{x_j}(\omega,t_i)|^{\al_1}\big]^{\frac{\al_2}{\al_1}}\Big]^{\frac1{\al_2}}
\lt (2L)^{\frac{1}{\alpha_1}} T^{\frac{1}{\alpha_2}} C_p(\om)~\text{for almost all}~\omega \in \Omega.
}
This shows that $RV^{\al_1,\al_2}_{[0,T]\times [-L,L]}(\ell) < \infty$ a.s. for any
$\al_1> \frac2{\beta-1}$ and $\al_2> \frac{2\beta}{\beta-1}$. The above argument also shows that $LV^{\al_2,\al_1}_{[0,T]\times [-L,L]}(\ell) < \infty~a.s$.
This allows us to conclude the proof.

\end{proof}
In the sequel, we denote $\Delta f(t,s;x,y):=f(t,x) - f(t,y) - \big( f(s,x) - f(s,y) \big); (t,s,x,y)\in [0,T]^2\times \mathbb{R}^2.$ A routine manipulation yields the following interpolation result. We omit the details of the proof.

\begin{lemma}\label{interpolation1}
Let $f:[0,T]\times [-M,M]\rightarrow \mathbb{R}$ be a function such that $LV^{a,b}_{[0,T]\times [-M,M]}(f) <\infty$ for $a,b\ge 1$. If $a< a'$ and $b' = \frac{a'}{a}b$, then

\aa{
LV^{a',b'}_{[0,T]\times [-M,M]}(f)\le \sup_{t,s\in [0,T];x,y\in [-M,M]}|\Delta f(t,s;x,y)|^{\frac{a'- a}{a'}}
\, \sup_{\pi}\Bigg[\sum_{j=1}^m\Big[\sum_{i=1}^n|\Delta_i\Delta_j f(t_i,x_j)|^a \Big]^{\frac{b}{a}} \Bigg]^{\frac{1}{b'}}.
}
Similarly, if $RV^{p,q}_{[0,T]\times [-M,M]}(f) <\infty$ for $p,q\ge 1$ and $p< p'$ and $q' = \frac{p'}{p}q$, then

\aa{
RV^{p',q'}_{[0,T]\times [-M,M]}(f)\le \sup_{t,s\in [0,T];x,y\in [-M,M]}|\Delta f(t,s;x,y)|^{\frac{p'- p}{p'}} \,
 \sup_{\pi}\Bigg[\sum_{i=1}^n\Big[\sum_{j=1}^m|\Delta_i\Delta_j f(t_i,x_j)|^p \Big]^{\frac{q}{p}} \Bigg]^{\frac{1}{q'}},
}
where sup varies over all partitions $\pi:=\{0=t_1 \le t_2 \le \ldots \le t_n=T\}\times \{-M=x_0\le x_1 \ldots \le x_m=M\}$ of $[0,T]\times [-M,M]$.
\end{lemma}
In the sequel, for a compact set $[0,T]\times [-M,M]$, we denote

$$\|f\|_{a,b; [0,T]\times [-M,M]}:= LV^{a,b}_{[0,T]\times [-M,M]}(f) + \|f(0,\cdot))\|_{[-M,M];b} +\|f(\cdot,-M)\|_{[0,T];a} + |f(0,-M)|$$
where $a,b\ge 1$. We define $LW_{a,b}([0,T]\times [-M,M])$ as the set of all functions $f:[0,T]\times [-M,M]\rightarrow \mathbb{R}$ such that $\|f\|_{a,b; [0,T]\times [-M,M]}< \infty$.

For $p,q\ge 1$, we also denote

$$|f|_{p,q; [0,T]\times [-M,M]}:= RV^{p,q}_{[0,T]\times [-M,M]}(f) + \|f(0,\cdot))\|_{[-M,M];q} +\|f(\cdot,-M)\|_{[0,T];p} + |f(0,-M)|$$
and $RW_{p,q}([0,T]\times [-M,M])$ is the set of all functions $f:[0,T]\times [-M,M]\rightarrow \mathbb{R}$ such that $|f|_{p,q; [0,T]\times [-M,M]}< \infty$. We refer the reader to Towgui~\cite{towghi} for details on this joint variation concept.

\

\noindent \textbf{Assumption D(i)} There exists $1\le a< \frac{2\beta}{\beta+1}$, such that $\sup_{x\in K}\|\nabla^w F_\cdot(^{x}c_{\cdot})\|_{a;[0,T]}< \infty$ for every $c\in C([0,T];\mathbb{R})$ and a compact subset $K\subset \mathbb{R}$.

\

\noindent \textbf{Assumption D(ii)} There exists $1\le b< \frac{2}{3-\beta}$, such that $\sup_{0\le t\le T}\|\nabla^w F_t(^{\cdot}c_t)\|_{b;[-M,M]}< \infty$ for every $c\in C([0,T];\mathbb{R})$ and $M>0$.

\

\begin{proposition}\label{corPQ}
Let $X$ be a stable symmetric process with index $1< \beta \le 2$. Assume that $F$ is a functional which satisfies Assumptions A1, A2, C and $D(i)$. If for each $c\in C([0,T];\mathbb{R})$, $(t,x)\mapsto(\nabla^w_x F_{t})(^{x}c_t) \in LW_{a,b}([0,T]\times [-M,M])$ for every $M>0$ with $1 \le a < \frac{2\beta}{\beta+1}$, $1\le b < \frac{2}{3-\beta}$ and $1 \le a\le b$, then

\begin{equation}\label{itotowgui}
F_t(X_t) = F_0(X_0)+ \int_0^t\nabla^hF_s(X_s)ds + \int_0^t\nabla^wF_s(X_s)dX(s) - \frac{1}{2}\int_{-\infty}^{+\infty}\int_0^t(\nabla^w_xF_s)(^{x}X_s)d_{(s,x)}\ell^x(s)~a.s
\end{equation}
for $0\le t\le T$.
\end{proposition}

\begin{proof}
In the sequel, we fix $M>0$ and to shorten notation, we omit $[0,T]\times [-M,M]$ and we write $\|\cdot\|_{a,b}$ and $LW_{a,b}$. We also write $\|\cdot\|_\gamma$ for the one-parameter H\"{o}lder norm over a compact set. Throughout this section, $C$ is a generic constant which may differ from line to line. From Boylan \cite{boylan}, we know that $\{\ell^x(s); (s,x)\in \mathbb{R}_+\times\mathbb{R}\}$ has jointly continuous paths a.s. From Lemma \ref{ltimejointvar} and Towgui~[Th 1.2 (b); \cite{towghi}], we know that the following integral process

\begin{equation}\label{intprocess}
\int_0^t\int_{-M}^M (\nabla^wF_s)(^{x}X_s)d_{(s,x)}\ell^x(s); 0\le t\le  T,
\end{equation}
exists if for any $c\in C([0,T];\mathbb{R})$, $(t,x)\mapsto(\nabla^w_x F_{t})(^{x}c_t) \in LW_{a,b}$ where

\begin{equation}\label{exp}
1 \le a < \frac{\alpha_2}{\alpha_2-1}, 1\le b< \frac{\alpha_1}{\alpha_1-1}\quad\text{and}\quad \alpha_1 > \frac{2}{\beta-1}, \alpha_2 > \frac{2\beta}{\beta-1}.
\end{equation}
Since $\frac{2}{3-\beta} = \sup\{\frac{\alpha_1}{\alpha_1-1}; \alpha_1 > \frac{2}{\beta-1}\}$ and $\frac{2\beta}{\beta+1} = \sup\{\frac{\alpha_2}{\alpha_2-1}; \alpha_2 > \frac{2\beta}{\beta - 1}\}$, then (\ref{intprocess}) exists whenever $\nabla^w F(c)\in LW_{a,b}$ for any $a < \frac{2\beta}{\beta+1}$ and $b < \frac{2}{3-\beta}$.

From Assumptions A1-A2 and Corollary \ref{Itoformulamoll}, the following decomposition holds

$$F^n_t(X^M_t) = F^n_0(X^M_0) + \int_0^t\nabla^h F^n_s(X^M_s)ds + \int_0^{t\wedge T_M}\nabla^v F^n_s(X_s)dX(s)
 -\int_0^{t\wedge T_M}\int_{-M}^M\nabla_x F^n(^{x}X_s)d_{(s,x)}\ell^x(s)$$
a.s.for $0\le t\le T, n\ge 1$. From Assumptions A1, A2 and C, we have already proved (See convergence in (\ref{step11}) and (\ref{step12})) that $\lim_{n\rightarrow \infty}F^n_t(X^M_t)= F_t(X^M_t)~a.s$ and

\begin{equation}\label{penparte}
\lim_{n\rightarrow \infty} \Big(\int_0^t\nabla^h F^n_s(X^M_s)ds + \int_0^{t\wedge T_M}\nabla^v F^n_s(X_s)dX(s)\Big) = \int_0^t\nabla^h F_s(X^M_s)ds + \int_0^{t\wedge T_M}\nabla^w F_s(X_s)dX(s)
\end{equation}
in probability for each $t\in [0,T]$. It only remains to check

\begin{equation}\label{ultparte}
\int_0^{t\wedge T_M} \int_{-M}^M  \nabla_x F^n_s(^{x}X_s)d_{(s,x)}\ell^x(s)\rightarrow \int_0^{t\wedge T_M} \int_{-M}^M  (\nabla^w_x F_s)(^{x}X_s)d_{(s,x)}\ell^x(s)
\end{equation}
a.s. as $n\rightarrow \infty$ for every $t\in [0,T]$. To shorten notation, let us denote $\Phi^n_s(x):=\nabla_x F^n_s(^{x} X_s) - (\nabla^w_x F_s)(^{x}X_s); (s,x)\in [0,T]\times[-M,M]$. Let us fix an arbitrary $t\in [0,T]$. In the sequel, we take $\varepsilon> 0$ small enough such that $a'  = a+\varepsilon $ and $b' = \frac{a'}{a}b$ satisfy $a'< \frac{2\beta}{\beta+1}$ and $b'< \frac{2}{3-\beta}$. We claim that

\begin{equation}\label{phiestimate1}
\|\Phi^n\|_{a',b'}\rightarrow 0~a.s~\text{as}~n\rightarrow \infty.
\end{equation}
A simple one parameter interpolation estimate (similar to Lemma \ref{interpolation1}) yields

\begin{equation}\label{phiestimate2}
\|\Phi^n_0\|_{b'}\le \sup_{x,y\in [-M,M]^2}|\Phi^n_0(x) - \Phi^n_0(y)|^{1-\frac{b}{b'}}\| \Phi^n_0\|_{b}^{\frac{b}{b'}}~a.s
\end{equation}
and

\begin{equation}\label{phiestimate3}
\|\Phi^n_\cdot(-M)\|_{a'}\le \sup_{s,t\in [0,T]^2}|\Phi^n_t(-M) - \Phi^n_s(-M)|^{1-\frac{a}{a'}}\| \Phi^n_0\|_{a}^{\frac{a}{a'}}~a.s
\end{equation}
where (\ref{step3}) yields $\sup_{x,y\in[M,M]^2}|\Phi^n_0(x) - \Phi^n_0(y)|^{1-\frac{b}{b'}}\rightarrow 0~a.s$ as $n\rightarrow \infty$. Moreover,

\aa{
\sum_{j=1}^m|\Delta_j \nabla_x F^n_0 (^{x_j}X_0)|^b \le \int_0^2 \rho(z) \sum_{j=1}^m|\nabla_j\nabla^w F_0 \big(^{x_j - \frac{z}{n}}X_0\big)|^bdz
\le C\|\nabla^w F_0 (^{\cdot}X_0)\|_b^b
}
so that $\sup_{n\ge 1}\|\nabla_x F^n_0 (^{\cdot}X_0)\|^b_b \le C\|\nabla^w F_0 (^{\cdot}X_0)\|^b_b~a.s$. Triangle inequality then allows us to conclude that $\sup_{n\ge 1}\|\Phi^n_0\|^{\frac{b}{b'}}_b\le C \|\nabla^w F_0 (^{\cdot}X_0)\|_b^{\frac{b}{b'}}~a.s$. Then (\ref{phiestimate2}) yields

\begin{equation}\label{phiestimate4}
\lim_{n\rightarrow + \infty}\|\Phi_0\|_{b'} = 0  \quad \text{a.s.}
\end{equation}
Similarly, by D(i),
\aa{
\sum_{i=1}^k|\Delta_i \nabla_x F^n_{t_i} (^{-M}X_{t_i})|^a&\le& \int_0^2 \rho(z) \sum_{i=1}^k|\Delta_i\nabla^w F_{t_i} \big(^{-M - \frac{z}{n}}X_{t_i}\big)|^adz
\le C\sup_{-2M\le x\le 0}\|\nabla^w F_\cdot (^{x}X_\cdot)\|_a^a
}
so that $\sup_{n\ge 1}\|F^n_\cdot(^{-M}X_\cdot)\|_a^a\le\sup_{-2M\le x\le 0}\|\nabla^w F_\cdot (^{x}X_\cdot)\|_a^a~a.s$. Triangle inequality, (\ref{phiestimate3}) and (\ref{step3}) yield

\begin{equation}\label{phiestimate5}
\lim_{n\rightarrow +\infty}\|\Phi^n_{\cdot}(-M)\|_{a'} = 0 \quad \text{a.s.}
\end{equation}
Summing up (\ref{phiestimate4}) and (\ref{phiestimate5}) and invoking again (\ref{step3}), we conclude that

\begin{equation}\label{phiestimate6}
\lim_{n\rightarrow +\infty}\Big(|\Phi^n_0(-M)| + \|\Phi^n_0\|_{b'} + \|\Phi^n(-M)\|_{a'}\Big)=0  \quad \text{a.s.}
\end{equation}

Now, we take $\frac{b}{a}\ge 1$ and Jensen inequality yields

\begin{equation}\label{phiestimate7}
\sum_{j=1}^m\Big(\sum_{i=1}^k |\Delta_i\Delta_j \nabla_x F^n_{t_i}(^{x_j}X_{t_i})|^a\Big)^{\frac{b}{a}}\le \int_0^2\sum_{j=1}^m\Big(\sum_{i=1}^k|\Delta_i\Delta_j(\nabla^w F_{t_i})(^{x_j - \frac{z}{n}}X_{t_i})|^a \Big)^{\frac{b}{a}}\rho(z)dz
\end{equation}
a.s. for every $n\ge 1$ and partition $\pi$ of $[0,T]\times [-M,M]$. Lemma \ref{interpolation1} yields

\aa{
LV^{a',b'}(\Phi^n)&\le \sup_{\substack{t,s\in [0,T] \\ x,y\in [-M,M]}}|\Delta \Phi^n(t,s;x,y)|^{\frac{a'- a}{a'}}
\times \sup_{\pi}\Bigg\{\Bigg[\sum_{j=1}^m\Big[\sum_{i=1}^k|\Delta_i\Delta_j \Phi^n_{t_i}(x_j)|^a \Big]^{\frac{b}{a}} \Bigg]^{\frac{1}{b'}} \Bigg\}\\
&\le C \hspace{-2mm}\sup_{\substack{t,s\in [0,T] \\ x,y\in [-M,M]}}|\Delta \Phi^n(t,s;x,y)|^{\frac{a'- a}{a'}}
\times\sup_{\pi}\Bigg\{\Bigg[\sum_{j=1}^m\Big[\sum_{i=1}^k|\Delta_i\Delta_j F^n_{t_i}(^{x_j}X_{t_i})|^a \Big]^{\frac{b}{a}} \Bigg]^{\frac{1}{b'}} \Bigg\}\\
&+  C \hspace{-2mm} \sup_{\substack{t,s\in [0,T] \\ x,y\in [-M,M]}}|\Delta \Phi^n(t,s;x,y)|^{\frac{a'- a}{a'}}
\times\sup_{\pi}\Bigg\{\Bigg[\sum_{j=1}^m\Big[\sum_{i=1}^k|\Delta_i\Delta_j (\nabla^wF_{t_i})(^{x_j}X_{t_i})|^a \Big]^{\frac{b}{a}} \Bigg]^{\frac{1}{b'}} \Bigg\}
}
a.s. for every $n\ge 1$. Then (\ref{step3}), (\ref{phiestimate6}) and (\ref{phiestimate7}) allow us to state that (\ref{phiestimate1}) holds true. Lastly, we take $(\alpha_1,\alpha_2)$ such that $a'< \frac{\alpha_2}{\alpha_2 - 1}, b'< \frac{\alpha_1}{\alpha_1 - 1}$ for $\alpha_1> \frac{2}{\beta-1}$ and $\alpha_2> \frac{2\beta}{\beta-1}$. By Th. 1.2 in Towghi \cite{towghi}, we know there exists a constant $C$ such that

\begin{equation}\label{towghiestimate1}
\Big|\int_0^{t\wedge T_M}\int_{-M}^M \Phi^n_s(x) d_{(s,x)}\ell^x(s)\Big|\le C \|\Phi^n\|_{(a',b')}\times LV^{\alpha_2,\alpha_1}(\ell)
\end{equation}
a.s. for every $n\ge 1$ and hence Lemma \ref{ltimejointvar}, (\ref{towghiestimate1}) and (\ref{phiestimate1}) allow us to conclude that decomposition (\ref{itotowgui}) holds over the stochastic set $[0,t\wedge T_M]$. By taking $M\rightarrow \infty$, we may conclude the proof.
\end{proof}
A complete similar proof also yields the symmetric result of Corollary \ref{corPQ} as follows.

\begin{corollary}\label{corPQA}
Let $X$ be a stable symmetric process with index $1< \beta \le 2$. Assume that $F$ is a functional which satisfies Assumptions A1, A2, C and $D(ii)$. If for each $c\in C([0,T];\mathbb{R})$, $(t,x)\mapsto(\nabla^w_x F_{t})(^{x}c_t) \in RW_{p,q}([0,T]\times [-M,M])$ for every $M>0$ with $1 \le p < \frac{2}{3-\beta}$, $1\le q < \frac{2\beta}{\beta+1}$ and $1 \le p\le q$, then

\aaa{
\label{f51}
F_t(X_t) = F_0(X_0)+ \int_0^t\nabla^hF_s(X_s)ds + \int_0^t\nabla^wF_s(X_s)dX(s) - \frac{1}{2}\int_{-\infty}^{+\infty}\int_0^t(\nabla^w_xF_s)(^{x}X_s)d_{(s,x)}\ell^x(s)~a.s
}
for $0\le t\le T$.
\end{corollary}

%\noindent \textbf{Example:} Path-dependent cylindrical functionals.
\begin{example}[Path-dependent cylindrical functionals]
\rm
Let $\{0=t_0 < t_1 < t_2 <\cdots < t_n = T\}$ be a partition of $[0,T]$. Consider a continuous function $f: \Rnu^{n}\to \Rnu$
weakly differentiable in each variable. Let us assume that for each $k$ and for each $i>k$,  the $i$th  weak partial derivative
\aaa{
\label{p-der1}
x\mto \nab^w_i f(c(t_1-), c(t_2-),\ldots, c(t_{k}-), x, \ldots, x_i, \ldots, x)\big|_{x_i = x},
}
evaluated at $x$,
is left continuous and is of bounded $q$-variation on $[-M,M]$ for each $M>0$ and for
some $q \in [1, \frac{2}{3-\beta})$.
For every $c\in\Lambda$, define the functional $F_t$ by the formulas:
\aaa{
\label{Ft}
F(c) = f(c(t_1-), c(t_2-), \ldots, c(t_n-)) \quad  \text{and} \quad F_t(c_t) = F(c_{t,T-t}).
}
Let us prove that It\^o's formula \eqref{f51} holds the functional $F_t$.
Let us notice that the functional $\mathcal F^x$, defined by \eqref{new-f}, takes the form:
\aaa{
\label{G-expli}
\mathcal F^x_t(c_t) = \sum_{k=0}^{n-1} f(c(t_1-), c(t_2-),\ldots, c(t_k-), x,\ldots, x)\, \ind_{\{t_k \lt  t < t_{k+1}\}}.
}
From this formula one immediately verifies that the family $\mathcal F^x_t(c_t)$
is state boundedness preserving and that $\nab^h \mathcal F^x_t(c_t) = 0$. For the weak
derivative we obtain:
\mmm{
\label{p-der}
\nab^w_x \mathcal F^x_t(c(t_1-), c(t_2-),\ldots, c(t_{k}-), x,\ldots, x) \\
 = \sum_{k=0}^{n-1} \ind_{\{t_{k} \lt  t < t_{k+1}\}} \sum_{i=k+1}^{n}
 \nab_i^w f(c(t_1-), c(t_2-),\ldots, c(t_{k}-), x,\ldots, x_i, \ldots, x)\big|_{x_i = x}.
}
This immediately
implies that Assumptions C and D(ii) are fulfilled. We also remark that
$(t,x)\mapsto(\nabla^w_x \mathcal F_t^x)(c_t) \in LW_{p,q}([0,T]\times [-M,M])$, where
$q$ is the same number as of the $q$-variation of \eqref{p-der1}, and $p$ is arbitrary.

We further note that the family $\mathcal F^x$ fails to be state-dependent
$\Lambda$-continuous. However, one immediately verifies that it is
state-dependent $\Lambda$-continuous on each interval $[t_{i-1}, t_i-\eps]$
for any sufficiently small $\eps$.
%Take $0\lt s<t \lt T$ and $c,\td c\in\Lambda$. Consider first the case when
%$t$ does not coincide with anyone of $t_i$'s. Suppose $t\in [t_k,t_{k+1})$. Clearly,
%we can choose $s\in [t_k,t_{k+1})$.
%We have:
%\aa{
%\mathcal F_s^x(\td c_s) = f(\td c(t_1-), \td c(t_2-),\ldots, \td c(t_k-), x,\ldots, x).
%}
%Note that the inequality $d_{\infty}((t,c), (s,\td c))<\eps$ implies
%that $t-s<\eps$ and $|\td c(s) - c(u)|<\eps$ for all $u \in (s,t]$. Now by continuity of $f$, $\mathcal F^x$ is
%$\Lambda$-continuous.
Therefore, on the interval $[0,t_1 -\eps]$ all assumptions of Proposition \ref{corPQ} are fulfilled, and therefore,
\mm{
F_t(X_t) = F_0(X_0)+ \int_0^{t_1-\eps}\nabla^hF_s(X_s)ds + \int_0^{t_1-\eps}\nabla^wF_s(X_s)dX(s) \\ - \frac{1}{2}\int_{-\infty}^{+\infty}\int_0^{t_1-\eps}(\nabla^w_xF_s)(^{x}X_s)d_{(s,x)}\ell^x(s).
}
Passing to the limit as $\eps\to 0$, we obtain  \eqref{itotowgui} for any $t\in [0,t_1]$.
By the same argument, \eqref{itotowgui} holds on each interval $[t_{i-1},t_i]$ with the initial condition $F_{t_{i-1}}(X_{t_{i-1}})$.
This implies \eqref{itotowgui} for every $t\in [0,T]$.
\end{example}

\begin{example}\label{FPS}
\rm
Let us now summarize Theorem \ref{thYoung}, Proposition \ref{corPQ} and Corollary \ref{corPQA}. One typical class of examples which can be treated by using the results of Sections \ref{funcsec1} and \ref{jointsection} is the following pathwise path-dependent version of the classical F\"ollmer-Protter-Shiryaev formula (\cite{follmer})

$$F_t(X_t) = \int_{-\infty}^{X(t)}Z_t(X_t; y)dy$$
where $Z = \{Z_t(\cdot; x): C([0,t];\mathbb{R})\rightarrow\mathbb{R}; 0\le t\le T, x\in\mathbb{R}\}$ can be chosen in such way that

$$\nabla^wF_t(^{x}X_t) = Z_t(X_t; x)\quad \text{and}\quad \nabla^h F_t(X_t) = \int_{-\infty}^{X(t)}\nabla^h Z_s(X_s;y)dy$$
satisfy the set of assumptions (A1, A2, C, D(i)) or (A1, A2, B). For a concrete case, see Example~\ref{expexample}. In this case, the following formula holds

$$F_t(X_t) = F_0(X_0) + \int_0^t\int_{-\infty}^{X(s)} \nabla^h Z_s(X_s;y)dyds + \int_0^t  Z_s(X_s;X(s))dX(s)  -
\frac{1}{2}\int_{-\infty}^{+\infty}\int_0^t Z_s(X_s;x)d_{(s,x)}\ell^x(s)$$
a.s. for $0\le t\le T.$
\end{example}

\

\textbf{Acknowledgment:} The authors would like Francesco Russo, Dorival Le\~ao and Estev\~ao Rosalino for stimulating discussions on the topic of this paper.

\end{document}